\documentclass[11pt]{amsart}

\usepackage{framed}
\usepackage{listings} 
\usepackage{amsmath,amssymb,amscd,amsthm,indentfirst,eufrak}
\usepackage{amsfonts}
\usepackage[all]{xy}
\usepackage{graphicx}
\usepackage{stmaryrd}
\usepackage{enumitem}
\usepackage[colorlinks,linkcolor=blue]{hyperref}
\usepackage[T1]{fontenc}

\usepackage[colorinlistoftodos,prependcaption]{todonotes}
\newcommand{\comment}[1]
{\todo[inline, linecolor=green,backgroundcolor=green!25,bordercolor=green,]{#1}}

\newtheorem{prop}{Proposition}[section]
\newtheorem{theorem}[prop]{Theorem}
\newtheorem{cor}[prop]{Corollary}

\newtheorem{lemma}[prop]{Lemma}

\theoremstyle{definition}
\newtheorem{defn}[prop]{Definition}
\newtheorem{example}[prop]{Example}
\newtheorem{void}[prop]{}

\newcommand{\B}[1]{\mathbb{#1}}

\def\L{\ensuremath{\mathcal{L}}}

\def\R{\ensuremath{\mathcal{R}}}
\def\T{\ensuremath{\mathcal{T}}}

\def\RR{\ensuremath{\mathbb{R}}}

\def\supp{\operatorname{supp}}

\def\defeq{\overset{\text{def}}{=}}

\newcommand{\OP}[1]{\operatorname{#1}}

\def\Ulpos{U_{\ell\text{-pos}}}

\author{Assaf Libman}
\address{Institute of Mathematics,
University of Aberdeen,
King's College,
Fraser Noble Building,
Aberdeen AB24 3UE,
United Kingdom}
\email{a.libman@abdn.ac.uk}


\begin{document}

\title[Polytopes associated with lattices of subsets]{Polytopes associated with lattices of subsets and maximising expectation of random variables}

\begin{abstract}
The present paper originated from a problem in Financial Mathematics concerned with calculating the value of a European call option based on multiple assets each following the binomial model.
The model led to an interesting family of polytopes $P(b)$ associated with the power-set $\L = \wp\{1,\dots,m\}$ and parameterized by $b \in \RR^m$, each of which is a collection of probability density function on $\L$.
For each non-empty $P(b)$ there results a family of probability measures on $\L^n$ and, given a function $F \colon \L^n \to \RR$, our goal is to find among these probability measures one which maximises (resp. minimises) the expectation of $F$.
In this paper we identify a family of such functions $F$, all of whose expectations are maximised (resp. minimised under some conditions) by the same {\em product} probability measure defined by a distinguished vertex of $P(b)$ called the supervertex (resp. the subvertex).
The pay-offs of European call options belong to this family of functions.
\end{abstract}

\maketitle

\section{Introduction and statement of results}
\label{S:main results}

This paper originated from a problem in Financial Mathematics which we describe in Section \ref{Int:FM motivation} below.
The combinatorial objects it led to are the subject of this paper.

\begin{void}{\bf Polytopes associated to the poset $2^m$.}
\label{SS:polytopes associates to 2m}
Let $\L=\{0,1\}^m \cong \wp\{ 1,2,\dots,m\}$ denote the poset of $m$-tuples of zeros and ones, i.e function $\lambda \colon \{1,\dots,m\} \to \{0,1\}$.
Let $\RR^\L$ denote the Euclidean space of dimension $2^m$ of all functions $x \colon \L \to \RR$, equipped with the standard basis $\{ e_\lambda\}_{\lambda \in \L}$ and inner product $\langle \ , \ \rangle$.

The unit simplex $\Delta(\L) \subseteq \RR^\L$ is the convex hull of $\{e_\lambda\}_{\lambda \in \L}$.
It is the set of all probability density functions on $\L$, see Section \ref{SS:vector spaces}. 
Vectors $f \in \RR^\L$ are viewed as random variables on $\L$ and it is clear that $\langle f ,x \rangle=E_x(f)$ is the expectation.

The assignment $e_\lambda \mapsto ((-1)^{\lambda(1)}, \dots, (-1)^{\lambda(m)})$ is a bijection between the vertices of the simplex $\Delta(\L)$ and the vertices of the $m$-dimensional cube $[-1,1]^m$.
There results a surjective linear map of polytopes $\Lambda \colon \Delta(\L) \to [-1,1]^m$ and we obtain a family of polytopes $P(b) \subseteq \Delta(\L)$ indexed by $b \in [-1,1]^m$, see Definition \ref{D:P(b)},
\[
P(b) \ \defeq \ \Lambda^{-1}(b).
\]
\end{void}

\begin{void}{\bf Maximizing expectations}
\label{SS:maximize expectation}
Fix some $n>0$.
Let $F \colon \L^n \to \RR$ be a function and $\Gamma \subseteq \Delta(\L^n)$ a compact connected subset of probability measures on $\L^n$.
Then $\{ E_P(F) : P \in \Gamma \}$ is a closed interval in $\RR$ and a fundamental question is to compute its end points
\[
F_{\min}(\Gamma)=\min \{ E_x(F) : x \in \Gamma\} \qquad \text{and} \qquad F_{\max}(\Gamma)=\max \{ E_x(F) : x \in \Gamma\}.
\]
In this generality the problem is hopeless unless we narrow down the choices for $\Gamma$ and $F$.

In Section \ref{SS:extension F tree} we will define the collections $\Gamma(\L^n,b)$ for every $b \in [-1,1]^m$.
We will introduce the collection of truncated $\ell$-positive functions $F \colon \L^n \to \RR$ in Definition \ref{D:symmetric L positive}.
The main result of this paper is Theorem \ref{T:Fmax Fmin is expectation of product measure} which shows that $F_{\max}(\Gamma)$ and $F_{\min}(\Gamma)$ are attained at a {\em product measure} on $\L^n$, explicitly described in terms of $b$.
\end{void}

\begin{void}{\bf Trees}\label{SS:extension F tree}
Fix some $n \geq 0$.
Let $\T$ denote the set of all words of length  at most $n$ in the alphabet $\L$.
It is partially ordered by $\tau \preceq \tau'$ if $\tau$ is a prefix of $\tau'$.
This renders $\T$ a directed tree with the empty word as its root.
The set of vertices at level $k$ is $\T_k=\L^k$ and $\L^n$ is the set of leaves.
For any $\omega \in \T$ set
\begin{equation}\label{E:def A_omega}
A_\omega = \{ \tau \in \L^n \ : \ \omega \preceq \tau\}.
\end{equation}
We will write $\omega \tau$ for the concatenation of words $\omega,\tau \in \T$.
Clearly, if $\omega \in \T_{n-k}$ then $A_\omega =\{\omega\tau: \tau \in \L^k\} \cong \L^k$.
Let $\T^*$ denote the set of words of length $<n$.
The set of {\em successors} of $\omega \in \T^*$, namely $\OP{succ}(\omega) = \{\omega \lambda : \lambda \in \L\}$, is canonically identified with $\L$.
We call $\T$ an {\em $\L$-labelled tree}.

A function $\Phi \colon \T^* \to \Delta(\L)$ is a choice of probability measures on $\OP{succ}(\omega)$ for every $\omega \in \T^*$. 
It gives rise to a probability density function $P(\Phi)$ on $\L^n$ 
\begin{equation}\label{D:P-Phi}
P(\Phi)(\lambda_1\cdots\lambda_n)=\prod_{k=1}^n \Phi(\lambda_1\cdots \lambda_{k-1})(\lambda_k).
\end{equation}
In fact, any probability measure on $\L^n$ arises in this way, see Proposition \ref{P:P(Phi) and its properties}.

It is natural to consider probability measures on $\L^n$ obtained from functions $\Phi$ with values in a given connected compact subset of $\Delta(\L)$.
Our interest is in $P(b) \subseteq \Delta(\L)$ and we define 
\begin{equation}\label{E:Gamma L^n b}
\Gamma(\L^n,b) = \{ P(\Phi) \ : \ \Phi \colon \T^* \to P(b)\}.
\end{equation}
Notice that $\Gamma(\L^n,b)$ is compact and connected since it is the image of $\prod_{\T^*}P(b)$.

There is an inductive procedure to compute $E_{P(\Phi)}(F)$ for $F \colon \L^n \to \RR$ and $P(\Phi) \in \Gamma(\L^n,b)$ by going down the levels of the tree $\T$.
Define by induction functions $F^{(k)}_\Phi \colon \T_{n-k} \to \RR$ where $F^{(0)}_\Phi=F$ and $F^{(k)}_\Phi(\omega) = E_{\Phi(\omega)}(F^{(k-1)}_{\Phi}|_{\OP{succ}(\omega)})$ for any $\omega \in \T_{n-k}$.
See Definition \ref{D:extension of F}.
Then that $F^{(n)}_\Phi(\emptyset)=E_{P(\Phi)}(F)$, see Proposition \ref{P:F^k as conditional probability}.
We can now describe an algorithm to find $F_{\max}(\Gamma)$ and $F_{\min}(\Gamma)$ where $\Gamma=\Gamma(\L^n,b)$.
\begin{void}{\bf Algorithm:}
Define functions $F^{(k)}_{\max} \colon \T_{n-k} \to \RR$, where $k \geq 0$, and $\Phi^{(k)}_{\max} \colon \T_{n-k} \to P(b)$ where $k \geq 1$, by induction as follows.
Set $F_{\max}^{(0)}=F$.
Assume $F^{(k-1)}_{\max}$ has been defined where $k \geq 1$.
Use the simplex method, or otherwise, to choose for any $\omega \in \T_{n-k}$ some $p \in P(b)$ which maximises $E_x(F_{\max}^{(k-1)}|_{\OP{succ}(\omega)})$ over $x \in P(b)$. 
Set $\Phi_{\max}^{(k)}(\omega)=p$ and let $F_{\max}^{(k)}(\omega)$ be this maximum expectation.

We obtain a function  $\Phi_{\max} \colon \T^* \to P(b)$, and one checks that $F_{\max}^{(k)}=F^{(k)}_{\Phi_{\max}}$ for all $k$.
By the monotonicity of the expectation it easily follows by induction that $F^{(k)}_\Phi(\omega) \leq F^{(k)}_{\max}(\omega)$ for any $\Phi \colon \T^* \to P(b)$.
Therefore $F_{\max}(\Gamma)=F_{\max}^{(n)}(\emptyset)$ and $P(\Phi_{\max})$ is the probability measure that realises the maximum.
An analogous algorithm computes $F_{\min}(\Gamma)$.
\end{void}
This calculation requires the simplex algorithm to be invoked $O\left(2^{m(n-1)}\right)$ times, once for each $\omega \in \T^*$.
This is exponential in $n$, the height of $\T$, and gives no insight to the problem.
The point of Theorem \ref{T:Fmax Fmin is expectation of product measure} is that for truncated $\ell$-positive functions $F$ the simplex algorithm can be avoided, and if in addition $F$ is symmetric then the calculation is polynomial in $n$.
\end{void}

\begin{void}{\bf Truncation and $\ell$-positive vectors}\label{SS:ell positive}
The {\em truncation} of $x \in \RR$ is $x^+=\max\{x,0\}$.
The truncation of $v \in \RR^\L$ is the vector $v^+$ with $v^+(\lambda) = v(\lambda)^+$.

Let $\ell_1,\dots,\ell_m \in \RR^\L$ be the rows of the matrix representing the linear map $\Lambda$ in Section \ref{SS:polytopes associates to 2m} and let $\ell_0 \in \RR^\L$ be the constant function with value $1$.
Let $U$ denote the subspace of $\RR^\L$ they span.
See Definition \ref{D:ell_i and U} and Example \ref{EX:L and L'} where the rows of the matrix $L$ are the vectors $\ell_i$ when $m=4$.
\end{void}

\begin{defn}\label{D:ell-positive vectors}
An {\em $\ell$-positive} vector in $U$ is a vector $u=\sum_{i=0}^m a_i \ell_i$ such that $a_1,\dots,a_m>0$ (and no condition on $a_0$). 
Let $U_{\ell\text{-pos}}$ be the set of these vectors.
The set of {\em truncated $\ell$-positive vectors} is $(\Ulpos)^+= \{ \sum_{i=1}^k u_i^+ : u_i \in \Ulpos, k \geq 0\}$. 
\end{defn}

\begin{defn}\label{D:symmetric L positive}
A function $F \colon \L^n \to \RR$ is called {\em symmetric} if the value of $F(\lambda_1\cdots\lambda_n)$ is independent of the order of the $\lambda_i$'s.
It is called {\em truncated $\ell$-positive} if the function $f \colon \lambda \mapsto F(\omega \lambda \tau)$ is an element of $(\Ulpos)^+$ for any words $\omega, \tau \in \T$ of total length $n-1$.
\end{defn}

Clearly, truncated $\ell$-positive functions have non-negative values.

\begin{void}{\bf The supervertex and the subvertex of $P(b)$.}\label{SS:supervertex and subvertex}
The main observation of this paper is that we can single out a {\em vertex} $q^* \in P(b)$, called the {\em supervertex} and a {\em vector} $q_* \in \RR^\L$ called the {\em subvertex} of $P(b)$, both described  purely in term of $b \in \RR^m$.
See Definitions \ref{D:supervertex - descending case}, \ref{D:supervertex - general case} and \ref{D:subvertex}.
To avoid confusion the reader is imperatively warned that the subvertex $q_*$ is only a {\em vector} in $\RR^\L$ and {\em need not} be in general an element of $P(b)$. 
Remarkably, when $q_* \in P(b)$ then it is a vertex of $P(b)$.
See Proposition \ref{P:q_* in P(b)}. 

The subvertex $q_*\in \RR^\L$ is supported by $\nu_0,\dots,\nu_m \in \L$ described in Definition \ref{D:nu_i} and $q_*(\nu_i)=b''(i)$ where $b''(i)$ are described in Definition \ref{D:b''}.

If $b(1)\geq \dots \geq b(m)$, the supervertex $q^* \in P(b)$ is supported by $\mu_0,\dots,\mu_m \in \L$ described in Definition \ref{D:mu_i} and $q^*(\mu_i)=b'(i)$ where $b'(i)$ are described in Definition \ref{D:bprime}.

The key results of this paper are Theorems \ref{T:max at supervertex} and \ref{T:min at subvertex} whose Corollary \ref{C:sum of ell-pos+ at supervertex} we restate here.
\end{void}

\begin{theorem}\label{Intro T:maximum at supervertex explicit}
Let $q^*$ and $q_*$ be the supervertex and subvertex of $P(b)$.
For any $u \in (U_{\ell\text{-pos}})^+$ we have $\langle u,q_* \rangle \geq 0$ and
\begin{eqnarray*}
&& \max \{ E_x(u) : x \in P(b) \} \ = \ E_{q^*}(u) = \langle u,q^* \rangle \\
&& \min \{ E_x(u) : x \in P(b) \} \ \geq \ \langle u,q_* \rangle
\end{eqnarray*}
If $\sum_{i=1}^m b(i) \leq 2-m$ then $q_* \in P(b)$, the inequality is an equality, and $\langle u,q_* \rangle=E_{q_*}(u)$.
\end{theorem}

Theorem \ref{Intro T:maximum at supervertex explicit} allows us to avoid appealing to the simplex method in the calculation of $F_{\max}(\Gamma)$ and $F_{\min}(\Gamma)$ in Section \ref{SS:extension F tree} for truncated $\ell$-positive functions $F$ and $\Gamma=\Gamma(\L^n,b)$.
Moreover, $F_{\max}(\Gamma)$ , and under some conditions $F_{\min}(\Gamma)$ are attained at a product measure on $\L^n$ defined by the supervertex and the subvertex.

For $F \colon \L^n \to \RR$ and $\omega \in \T_{n-k}$ identify $F|_{A_\omega}$ with the function $F_{\omega-} \colon \L^k \to \RR$ defined by
\begin{equation}\label{E:F_omega-}
F_{\omega-}(\tau) = F(\omega\tau), \qquad (\tau \in \L^k).
\end{equation}
If $\omega$ is the empty word then $F_{\omega-}=F$.

\begin{theorem}\label{T:Fmax Fmin is expectation of product measure}
Let $q^*$ and $q_*$ be the supervertex and subvertex of $P(b)$ and set $\Gamma=\Gamma(\L^n,b)$.
Let $F \colon \L^n \to \RR$ be a truncated $\ell$-positive function.
Then for any $P \in \Gamma$ and any $\omega \in \T_{n-k}$ such that $P(A_\omega)>0$
\[
E_P(F|A_\omega) \leq E_{q^*}(F_{\omega-}).
\]
In particular $F_{\max}(\Gamma)=E_{q^*}(F)$.

If $q_* \in P(b)$, which is equivalent to the condition $\tfrac{1}{m}\sum_{i=1}^m b(i) \leq \tfrac{2}{m}-1$, then
\[
E_P(F|A_\omega) \geq E_{q_*}(F_{\omega-}).
\]
In particular $F_{\min}(\Gamma)=E_{q_*}(F)$.
\end{theorem}

When $F$ is in addition symmetric we can give highly computable formulas for the right hand sides of the inequalities in Theorem \ref{T:Fmax Fmin is expectation of product measure}.
For any $\lambda \in \L$ let $\lambda^k$ denote the word $\lambda \cdots \lambda$ of length $k$.
Recall the description of $q^*$ and $q_*$ in Section \ref{SS:supervertex and subvertex}.
For $p \in \Delta(\L)$, let $p$ also denote the product measure on $\L^k$ for any $k \geq 0$.

\begin{prop}\label{P:symmetric polynomial complexity}
Assume the hypotheses of Theorem \ref{T:Fmax Fmin is expectation of product measure}.
Assume further that $F$ is symmetric.
Then for any  $\omega \in \T_{n-k}$
\begin{enumerate}[label=(\alph*)]
\item
If $b$ is decreasing i.e $b(1) \geq \dots \geq b(m)$ then
\[
E_{q^*}(F_{\omega-}) =
\sum_{i_0+\dots+i_m=k} \frac{k!}{i_0!\cdots i_m!} \cdot b'(0)^{i_0}\cdots b'(m)^{i_m} \cdot  F(\omega \mu_0^{i_0} \cdots \mu_m^{i_m}).
\]
\item
If $q_* \in P(b)$ then
\[
E_{q_*}(F_{\omega-}) =
\sum_{i_0+\dots+i_m=k} \frac{k!}{i_0!\cdots i_m!}\cdot  b''(0)^{i_0}\cdots b''(m)^{i_m} \cdot F(\omega \nu_0^{i_0} \cdots \nu_m^{i_m}).
\]
\end{enumerate}
\end{prop}

The advantage in $F$ being  symmetric is evident: the complexity of the computation, i.e the number of terms in the sums computing $F_{\max}(\Gamma)$ and $F_{\min}(\Gamma)$, is ${n+m} \choose m$, polynomial in $n$ rather than exponential (take $k=n$ and $\omega$ empty).
These results leave something to be desired, though.
Namely are there any interesting symmetric truncated $\ell$-positive functions?
In addition, the condition $q_* \in P(b)$ is unreasonably strong in practical applications. 

\begin{defn}\label{D:European}
A function $F \colon \L^n \to \RR$ is called {\em European} if there are $\ell$-positive vectors $u_1,\dots,u_r$ such that $u_i=u_i^+$ for all $i$, and numbers $s_1,\dots,s_r, C \geq 0$ such that
\[
F(\lambda_1\cdots \lambda_n) = \Big(\sum_{j=1}^r s_j \cdot  u_j(\lambda_1) \cdots u_j(\lambda_n) - C\Big)^+.
\]
\end{defn}

The terminology is inspired by the Financial Mathematics model in Section \ref{Int:FM motivation}.
European functions exist in abundance as we explain in Section \ref{S:truncated ell positive functions}.
They are symmetric truncated $\ell$-positive by Proposition \ref{P:cooking up F from ui} and therefore $F_{\max}(\Gamma)$ can be computed for them by Theorem \ref{T:Fmax Fmin is expectation of product measure}.
The next theorem gives a lower bound for $F_{\min}(\Gamma)$ for European functions. 
The bound tends to be very crude, though. 

\begin{theorem}\label{T:bounding Fmin below}
Let $P(b)$ be non-empty for some $b \in \RR^m$ and set $\Gamma=\Gamma(\L^n,b)$.
Let $F \colon \L^n \to \RR$ be a European function defined and let $\beta(0),\dots,\beta(m)$ and $\alpha_j(0),\dots,\alpha_j(m)$ where $j=1,\dots,r$ be the numbers in Definition \ref{D:Minimizer of European} associated to $F$.
Then
\[
F_{\min}(\Gamma) \geq
\sum_{k_0+\dots+k_m=n} \frac{n!}{k_0! \cdots k_m!} \cdot \beta(0)^{k_0} \cdots \beta(m)^{k_m} \big(\sum_{j=1}^r s_j \cdot \alpha_j(0)^{k_0} \cdots \alpha_j(m)^{k_m}-C\big)^+.
\]
\end{theorem}

\begin{void}
\label{Int:FM motivation}
{\bf Financial Mathematics motivation.} 
In this section we describe the problem in Financial Mathematics that has driven this project.
This section is aimed for the non-experts and we will therefore avoid Financial Mathematics jargon and (deliberately) use non-standard terminology.
A full account can be found in \cite{KLS2} and background material in \cite[Chap. 3]{Pliska}.

An example of a {\em discrete time market model} is a finite probability space $(\Omega,P)$ together with a set $\B T=\{0,1,\dots,n\}$ representing (discrete) time.
It is assumed that $P(A)=0$ if and only if $A=\emptyset$.
An {\em asset} is a sequence of random variables $X(0),\dots,X(n)$ indexed by $\B T$ such that $X(0)$ is a constant random variable representing the fact that its price at time $0$ is known.
A discrete market model is specified by assets $X_1,\dots,X_r$, each is a random process indexed by $\B T$. 
One of those assets is assumed to be a {\em bond process}, denoted by $B$. Thus $B(k)$ is the value at time $k\in \B T$ of a unit of money deposited in a savings account at time $k=0$. 
The ratio $r(k) = (B(k)-B(k-1))/B(k-1)$ is the interest rate which in the model we describe below is assumed to be constant, i.e $B(k)=R^k$ for a fixed interest rate $R\geq1$.

A {\em portfolio} is a vector $(x_1,\dots,x_r) \in \RR^r$ and its {\em value} is $V_x=\sum_i x_i X_i$.
A portfolio $x$ is called an {\em arbitrage} if
\begin{equation}\label{E:arbitrage}
\begin{array}{ll}
V_x(0)=0 \\
V_x(n)\geq 0 \\
E(V_x(n))>0.
\end{array}
\end{equation}
It is generally assumed that financial models do not have arbitrage portfolios.
It models an idealisation of reality in which no one should be able to make money
out of nothing with no risk of making loss; see \cite[Chapter 1]{MR2200584}.

For what follows we fix some assets $S_1,\dots,S_m$ which we call shares of stock.
A {\em European call option} is a contract made at time $k=0$ which gives its holder the right, but not an obligation, to buy at time $n$ a portfolio $x=(x_1,\ldots,x_m)$ whose value is $V_x(n)=\sum_i x_i S_i(n)$ for a given price $C$ set in the contract.
If $C<V_x(n)$ then the holder will exercise the option and buy the portfolio (for $C$), sell it for $V_x(n)$ and make a profit $V_x(n)-C$. 
If $C\geq V_x(n)$ the holder will do nothing. 
In other words, at time $n$ the holder of an option will make a profit $F=\max\{V_x(n)-C,0\}$.
The random variable $F$ is called the {\em pay-off}. 

Of course, the option is itself an asset $H$ for which $H(n)=F$.
One of the basic problems in Financial Mathematics is to determine $H(0)$, namely the price of the option at time $k=0$, as well as its values at any given time $k\in \B T$ so that an arbitrage does not occur.
Such a value is called {\em rational}.

A {\em risk-neutral} probability measure $P_*$ on $\Omega$ is a martingale measure with respect to the random processes $S_1,\dots,S_m$ with $P_*(\omega)>0$ for all $\omega \in \Omega)$, \cite[pp. 93]{Pliska}.
That is, the conditional expectation of $S_i(k+\ell)$ given the event that the values of $S_1(t), \dots, S_m(t)$ are known for all $0 \leq t \leq k$ is $R^\ell$ times the known value of $S_i(k)$.
To be more precise, for any $\omega \in \Omega$ consider the event $E_{k,\omega} = \bigcap_{t=0}^k \bigcap_{i=1}^m \{ S_i(t)=S_i(t)(\omega)\}$.
We require that
\begin{equation}\label{Eq:martingale}
E_{P_*}\big(S_i(k+\ell) \ \big| E_{k,\omega} \ \big) = R^\ell S_i(k)(\omega).
\end{equation}
The set $\Gamma_*$ of all the risk-neutral probability measures on $\Omega$ is therefore the interior of a convex bounded polytope equal to the intersection of the simplex of all probability measures on $\Omega$ with the affine subspace defined by the system of linear equations \eqref{Eq:martingale}.
Throughout we assume that $\Gamma_* \neq \emptyset$. 
This is equivalent to the absence of arbitrage portfolios; see \cite[(3.19)]{Pliska} or \cite[Theorem 1.6.1]{MR2200584} for a more general statement.

It is a fundamental result that the rational value $H$ of an option is the conditional expectation of the pay-off with respect to a risk-neutral probability measure \cite[Theorem 2.4.1]{MR2200584}, more precisely
\begin{equation}\label{E:F from f by E*}
H(k)(\omega) = E_{P_*} \big( F \ \big| \ E_{k,\omega} \big).
\end{equation}
The risk-neutral probability $P_*$ is not unique, and therefore neither is $H$.
We will write $H_{P_*}$ for the rational value in \eqref{E:F from f by E*}.
Thus, the set of all rational prices at time $k=0$, namely $\{ E_{P_*}(F) | P_* \in \Gamma_*\}$, forms an open interval $(F_{\min}, F_{\max})$. 
The question that has driven this paper was to find the values of $F_{\min}$ and $F_{\max}$ of a European call option in a model in which the shares $S_i$ follow binomial processes. 

\noindent
{\bf Specification of the model:}
For any $1 \leq i \leq m$ we fix $0<D_i<R<U_i$.
We also fix $S_i(0)>0$, prices at time $0$.
Each $S_i$ follows a binomial process, namely at time $k$ one flips a coin, possibly unfair, 
and according to the result $\epsilon=0,1$ the value of $S_i(k)$ is multiplied by either $D_i$ (if $\epsilon=1$) or by $U_i$ (if $\epsilon=0$).
Thus $S_i(k+1)=S_i(k) \cdot D_i^\epsilon U_i^{1-\epsilon}$.

The sample space suitable to describe this process is $\Omega=(\{0,1\}^m)^n$ which in the notation of Section \ref{SS:polytopes associates to 2m} is $\L^n$. 
Then $S_i(k)(\lambda_1\cdots\lambda_n)=S_i(0) \cdot D_i^{\sum_{j=1}^k\lambda_j(i)}U_i^{k-\sum_{j=1}^k\lambda_j(i)}$. 
The pay-off (at time $n$) in this model is therefore the random variable
\begin{equation}\label{E:pay off in binomial model}
F(\lambda_1\cdots\lambda_n) = \max \left\{ 0 \ , \ \sum_{i=1}^m S_i(0) \cdot D_i^{\sum_{j=1}^n \lambda_j(i)}\cdot U_i^{n-\sum_{j=1}^n \lambda_j(i)} \, - \, C\right\}.
\end{equation}

\begin{prop}\label{P:pay off is ell positive}
In the multi-step binomial model of a European call option described above, the pay-off \eqref{E:pay off in binomial model}  is a European function (in the sense of Definition \ref{D:European}).
\end{prop}

It is clear that for any $\lambda_1\dots\lambda_n$ in the sample space $\L^n$ the value of $S_i(k)$ is determined by $\omega=\lambda_1\dots\lambda_k$.
Thus, in the notation of Section \ref{SS:extension F tree}, $S_1(k),\dots,S_m(k)$ are constant on the events $A_\omega$, and hence so is $H(k)$.
In addition, it is easily verified that $\omega$ is determined by the values of $S_1(t),\dots,S_m(t)$ where $0 \leq t \leq k$.
Therefore $E_{k,\lambda_1\dots\lambda_n}=A_\omega$. 
Also, one checks that equations \eqref{Eq:martingale} for $\ell \geq 2$ are a consequence of those with $\ell=1$.
From the latter equations one checks that $P_*$ solves \eqref{Eq:martingale} and has no null-sets if and only if, with the notation of \eqref{D:P-Phi}, $P_*=P(\Phi)$ where the values of $\Phi \colon \T^* \to \Delta(\L)$ are in the set $Q \subseteq \Delta(\L)$ of the solutions of the system 
\begin{equation}\label{E:one step martingale}
E_x(S_i(1))=RS_i(0), \qquad \text{$1 \leq i \leq m$ and $x \in \Delta(\L)$ and $x(\lambda)>0$.}
\end{equation}
Then $Q$ is the interior of the polytope in $\Delta(\L)$ which is the preimage of $(R,\dots,R) \in \RR^m$ under a linear map that sends the vertices $e_\lambda$ of $\Delta(\L)$ to the vertices of the cube $[D_1,U_1] \times \dots \times [D_m,U_m]$ in $\RR^m$.
One checks, as we do in \cite{KLS2}, that $Q$ is the interior of $P(b)$ from Section \ref{SS:polytopes associates to 2m} where 
\begin{equation}\label{E:b from U R D}
b(i)=\frac{2R - U_i-D_i}{U_i-D_i}
\end{equation}
and that  $|b(i)| \leq 1$ by the assumption that $D_i<R<U_i$, so $P(b) \neq \emptyset$. 
The crux is now that
\[
\overline{\Gamma_*} = \Gamma(\L^n,b).
\]
By possibly reordering the shares $S_i$ we can ensure that $b$ is decreasing, i.e $b(1) \geq \dots \geq b(m)$.
Recall from Section \ref{SS:supervertex and subvertex} the description of the supervertex and the subvertex of $P(b)$.
We are now able to describe the interval of rational values of a European call option in this model.

\begin{theorem}\label{T:pay off rational values}
Let $F$ be the pay-off in \eqref{E:pay off in binomial model} and $H$ its rational value.
Consider some $\omega=\lambda_1\dots\lambda_n \in \L^n$ and $1 \leq k \leq n$.
Set $\theta=\lambda_1 \dots \lambda_{n-k}$.
Then 
\[
\sup_{P_* \in \Gamma_*} H_{P_*}(n-k)(\omega) = E_{q^*}(F_{\theta-}) = 
\sum_{i_0+\dots+i_m=k} \frac{k!}{i_0!\cdots i_m!} \cdot b'(0)^{i_0}\cdots b'(m)^{i_m} \cdot  F(\theta \mu_0^{i_0} \cdots \mu_m^{i_m}).
\]
If $\tfrac{1}{m}\sum_{i=1}^m b(i) \leq \tfrac{2}{m}-1$ then 
\[
\inf_{P_* \in \Gamma_*} H_{P_*}(n-k)(\omega) = E_{q_*}(F_{\theta-}) = 
\sum_{i_0+\dots+i_m=k} \frac{k!}{i_0!\cdots i_m!}\cdot  b''(0)^{i_0}\cdots b''(m)^{i_m} \cdot F(\theta \nu_0^{i_0} \cdots \nu_m^{i_m}).
\]
\end{theorem}

Hence, the maximal value of the pay-off at time $0$ in the $n$-step model, $F_{\max}$, is computed by a product measure obtained from a  martingale measure which can be computed explicitly from the parameters of the model.
Under some assumptions the same holds for $F_{\min}$.
These results generalise ones obtained in \cite{MR2283328} 
when $m=2$ (in which case $\dim P(b) \leq 1$ namely it is generically an interval).
This is not only a surprising result, but also has significant practical consequences since it dramatically reduces the computational complexity of $F_{\max}$ to $O(n^m)$.
\end{void}

\subsection*{Acknowledgements} 
I thank Jarek Kedra for an abundance of helpful discussions and ideas.
We both thank Victoria Steblovskaya for drawing our attention to this circle of problems and for discussions.

\section{The poset $\L$, the vectors $\ell_i$ and the polytopes $P(b)$}

\begin{void}\label{SS:vector spaces}
Given a finite set $\Omega$ let $\RR^\Omega$ denote the linear space of functions $x \colon \Omega \to \RR$ equipped with the standard basis $\{e_\omega\}_{\omega \in \Omega}$  and the standard inner product $\langle x,y\rangle = \sum_{\omega \in \Omega} x(\omega)y(\omega)$.
The {\em support} of $x \in \RR^\Omega$ is 
\[
\supp(x) = \{ \omega \in \Omega : x(\omega)  \neq 0 \}.
\]
The unit simplex $\Delta(\Omega)$ in $\RR^\Omega$ is the set of all probability density functions on $\Omega$
\[
\Delta(\Omega) = \left\{ x \in \RR^\Omega \ : \ \sum_{\omega \in \Omega} x(\omega)=1 \text{ and } x(\omega) \geq 0 \right\}.
\]
The {\em truncation} of $x \in \RR^\Omega$ is the vector $x^+ \in \RR^\Omega$ defined by $x^+(\omega)=x(\omega)^+=\max\{ x(\omega),0\}$.

For $m \geq 1$ set $[m]=\{1,2,\dots,m\}$ and $[m]_0=\{0,1,\dots,m\}$.
Throughout we will identify $\RR^m$ with $\RR^{[m]}$ and $\RR^{m+1}$ with $\RR^{[m]_0}$.
\end{void}

\begin{void}\label{SS:L}
{\em The poset $\L$.}
Let $\L$ denote the set of functions $\lambda \colon [m] \to \{0,1\}$ identified with the poset $\wp([m])$.
Thus, $\lambda \preceq \lambda'$ if $\supp(\lambda) \subseteq \supp(\lambda')$. 
It will be convenient to regard $\lambda$ as having domain $\{0,\dots,m+1\}$ and agree throughout that
\[
\lambda(0)=0, \quad \text{ and} \qquad \lambda(m+1)=1.
\]
\end{void}

\begin{void}\label{SS:action Sigma_m on R^L}
{\em Action of the symmetric group $\Sigma_m$.}
Let $\Sigma_m$ act on $[m]$ in the natural way.
Any $\sigma \in \Sigma_m$ gives rise to functions $\sigma_* \colon \L \to \L$ and  $\sigma_* \colon \RR^{m+1} \to \RR^{m+1}$, both abusively denoted by $\sigma_*$, defined by
\[
\sigma_*(\lambda) = \lambda \circ \sigma^{-1} \qquad \text{ and } \qquad \sigma_*(b) = b \circ \sigma^{-1}
\]
with the understanding that $\sigma$ acts on $[m]_0$ by fixing $0$ so $\sigma_*$ acts on $\RR^{m+1}$ by fixing the $0$th entry.
In turn, we obtain $(\sigma_*)_* \colon \RR^\L \to \RR^\L$ which we abusively also denote by $\sigma_*$ 
\[
\sigma_*(x)(\lambda) = x(\sigma_*{}^{-1}(\lambda))=x(\lambda \circ \sigma).
\]
Thus, every $\sigma \in \Sigma_m$ acts on $\RR^\L$ as a permutation matrix, hence an orthogonal transformation.
\end{void}

\begin{defn}\label{D:ell_i and U}
{\em (The vectors $\ell_0,\dots,\ell_m$).}
For every $0 \leq i \leq m$ let $\ell_i \in \RR^\L$ be the vector
\[
\ell_i(\lambda)=(-1)^{\lambda(i)}.
\]
Let $U$ be the subspace of $\RR^\L$ spanned by $\ell_0,\dots,\ell_m$.
\end{defn}

Notice that with the convention $\lambda(0)=0$ in Section \ref{SS:L}, $\ell_0$ is the constant function with value $1$.  
See Example \ref{EX:L and L'} where the rows of the matrix $L$ are the vectors $\ell_0,\dots,\ell_m$ for $m=4$.
It is an elementary exercise to verify that $\ell_0,\dots,\ell_m$ is an orthogonal system with respect to the standard inner product in $\RR^\L$, indeed $\langle \ell_i,\ell_j \rangle = 2^m \delta_{i,j}$.

\begin{prop}\label{P:Sigma_m permutes ell_i}
$\Sigma_m$ permutes $\ell_0,\dots,\ell_m$ in the natural way and leaves $\ell_0$ fixed.
That is, $\sigma_*(\ell_0)=\ell_0$ for any $\sigma \in \Sigma_m$ and for any $1 \leq i \leq m$
\[
\sigma_*(\ell_i) =\ell_{\sigma(i)}.
\]
\end{prop}

\begin{proof}
$\sigma_*(\ell_i)(\lambda) = \ell_i(\lambda\circ\sigma) = (-1)^{(\lambda\circ\sigma)(i)} =\ell_{\sigma(i)}(\lambda)$ for all $\lambda \in \L$.
\end{proof}

By definition of the unit simplex, $x \in \Delta(\L)$ if and only if $\langle \ell_0,x \rangle =1$ and $x(\lambda) \geq 0$ for all $\lambda \in \L$.
This justifies the following definition.

%
%
%

\begin{defn}\label{D:P(b)}
{\em (The polytopes $P(b)$).}
Let $L \colon \RR^\L \to \RR^{m+1}$ be the linear transformation
\[
L(x)(i) = \langle \ell_i, x\rangle, \qquad (0 \leq i \leq m). 
\]
For any $b \in \RR^m$ 
let $P(b) \subseteq \Delta(\L)$ be the polytope
\[
P(b) \, = \, L^{-1}\left((\begin{smallmatrix}1\\b\end{smallmatrix})\right) \, \cap \, \{ x \in \RR^\L : x(\lambda) \geq 0\}.
\]
\end{defn}

Notice that $P(b)$ is the intersection of an affine subspace of $\RR^\L$ with half-space, hence it is a polytope \cite[\S 8]{MR683612}.
Also, $L(e_\lambda)=(\ell_0(\lambda), \dots,\ell_m(\lambda))$ are the vertices of the $m$-dimensional cube $\{1\} \times [-1,1]^m$ in $\RR^{m+1}$.
Therefore, as in Section \ref{SS:polytopes associates to 2m}, $L$ restricts to a surjective linear map of polytopes $\Lambda \colon \Delta(\L) \to [-1,1]^m$ and $P(b)=\Lambda^{-1}(b)$.
The next Proposition follows.

\begin{prop}\label{P:non empty P(b)}
Any two polytopes $P(b)$ and $P(b')$ are either equal or disjoint.
The polytope $P(b)$ is not empty if and only if $\|b\|_\infty \leq 1$, namely $|b(i)| \leq 1$ for all $1 \leq i \leq m$.
\hfill $\Box$
\end{prop}

\begin{prop}\label{P:P(b)-->P(sigma(b))}
For any $\sigma \in \Sigma_m$ and any $b \in \RR^m$ the linear map $\sigma_* \colon \RR^\L \to \RR^\L$ restricts to an isomorphism of polytopes $\sigma_* \colon P(b) \to P(\sigma_*(b))$.
\end{prop}

\begin{proof}
Regard $b$ as a vector in $\RR^{m+1}$ with $b(0)=1$.
Since $\Sigma_m$ acts by orthogonal transformations on $\RR^\L$, for any $x \in P(b)$ 
\[
\langle \ell_i , \sigma_*(x) \rangle =
\langle \sigma_*^{-1} \ell_i , x \rangle =
\langle \ell_{\sigma^{-1}(i)} , x \rangle =
b(\sigma^{-1}(i))=
\sigma_*(b)(i).
\]
It easily follows that $\sigma_*(P(b)) = P(\sigma_*(b))$.
\end{proof}

Set $\mathcal{P}=\{ P(b)  : \text{$P(b)$ is not empty}\}$.
Propositions \ref{P:non empty P(b)} and \ref{P:P(b)-->P(sigma(b))} readily imply

\begin{cor}\label{C:Sigma equivariant action on P} 
The assignment $b \mapsto P(b)$ induces a $\Sigma_m$-equivariant bijection $[-1,1]^m \cong \mathcal{P}$.
\hfill $\Box$
\end{cor}

%
%
\section{The supervertex of $P(b)$}
\label{S:supervertex}

Recall the vectors $\ell_0,\dots,\ell_m$ and the subspace $U$ from Definition \ref{D:ell_i and U}.

\begin{defn}\label{D:define ell'} 
Let $\ell'_0,\dots,\ell'_m \in \RR^\L$ be the vectors
\begin{eqnarray*}
&& \ell'_i = \tfrac{1}{2} (\ell_i - \ell_{i+1}), \qquad 0 \leq i \leq m-1 \\
&& \ell'_m = \tfrac{1}{2} (\ell_0 + \ell_m).
\end{eqnarray*}
\end{defn}

One checks that $\ell_k = -\sum_{i=0}^{k-1} \ell_i' + \sum_{i=k}^m \ell_i'$ for every $0 \leq k \leq m$, thus $\ell_0',\dots,\ell'_m$ is a basis for $U$.


\begin{defn}\label{D:bprime}
Given $b \in \RR^m$ write $b(0)=1$ and $b(m+1)=-1$.
Let $b' \in \RR^{m+1}$ be the vector
\[
b'(i)=\frac{b(i)-b(i+1)}{2}, \qquad 0 \leq i \leq m.
\]
\end{defn}

\begin{prop}\label{P:P(b) in basis ell'}
Let $x \in \RR^\L$.
Then $x \in P(b)$ if and only if $\langle \ell'_i,x \rangle=b'(i)$ for all $0 \leq i \leq m$ and $x(\lambda) \geq 0$ for all $\lambda \in \L$.
\end{prop}

\begin{proof}
Similar to the linear map $L$ in Definition \ref{D:P(b)} let $L' \colon \RR^\L \to \RR^{m+1}$ be the linear transformation $L'(x)(i)=\langle \ell'_i,x \rangle$ where $0 \leq i \leq m$.
Since $\ell'_0,\dots,\ell'_m$ is a basis for $U$ it follows that $\ker(L')=U^\perp=\ker(L)$.
If $v \in L^{-1}\left((\begin{smallmatrix}1\\b\end{smallmatrix})\right)$ then one checks using Definitions \ref{D:define ell'} and \ref{D:bprime} that $L'(v)=b'$.
Thus, $L^{-1}\left((\begin{smallmatrix}1\\b\end{smallmatrix})\right)=v+U^\perp=L'{}^{-1}(b')$ and this completes the proof.
\end{proof}


\begin{example}\label{EX:L and L'}
Suppose that $m=4$.
We write the vectors $\ell_i \in \RR^\L$ from Definition \ref{D:ell_i and U} as the rows of the matrix $L$ below, where ``$-$'' denotes $-1$ and the columns of the matrix are indexed by the elements $\lambda$ of $\L=2^{[4]}$ ordered lexicographically.
\[
L=
\left[
\begin{array}{cccccccccccccccc}
{\scriptscriptstyle 0000} &
{\scriptscriptstyle 0001} &
{\scriptscriptstyle 0010} &
{\scriptscriptstyle 0011} &
{\scriptscriptstyle 0100} &
{\scriptscriptstyle 0101} &
{\scriptscriptstyle 0110} &
{\scriptscriptstyle 0111} &
{\scriptscriptstyle 1000} &
{\scriptscriptstyle 1001} &
{\scriptscriptstyle 1010} &
{\scriptscriptstyle 1011} &
{\scriptscriptstyle 1100} &
{\scriptscriptstyle 1101} &
{\scriptscriptstyle 1110} &
{\scriptscriptstyle 1111} \\
1 & 1 & 1 & 1 & 1 & 1 & 1 & 1 & 1 & 1 & 1 & 1 & 1 & 1 & 1 & 1 \\
1 & 1 & 1 & 1 & 1 & 1 & 1 & 1 & - & - & - & - & - & - & - & - \\
1 & 1 & 1 & 1 & - & - & - & - & 1 & 1 & 1 & 1 & - & - & - & - \\
1 & 1 & - & - & 1 & 1 & - & - & 1 & 1 & - & - & 1 & 1 & - & - \\
1 & - & 1 & - & 1 & - & 1 & - & 1 & - & 1 & - & 1 & - & 1 & - \\
\end{array}
\right]
\]
Consider a vector $b \in \RR^m$ where $|b(i)| \leq 1$ for all $1 \leq i \leq m$.
The polytope $P(b)$ is the set of solutions of 
\[
Lx=(\begin{smallmatrix}1\\b\end{smallmatrix}), \qquad x(\lambda) \geq 0.
\]
Consider the following $5 \times 5$ matrix $T$ and its inverse.
These are the transition matrices between the bases $\ell_0,\dots,\ell_m$ and $\ell'_0,\dots,\ell_m'$ of $U$.
See Definition \ref{D:define ell'}.
\[
T=\frac{1}{2} \cdot \begin{bmatrix}
1 & - & 0 & 0 & 0 \\ 
0 & 1 & - & 0 & 0 \\ 
0 & 0 & 1 & - & 0 \\  
0 & 0 & 0 & 1 & - \\ 
1 & 0 & 0 & 0 & 1  
\end{bmatrix}
\qquad
T^{-1} =
\begin{bmatrix}
1 & 1 & 1 & 1 & 1 \\
- & 1 & 1 & 1 & 1 \\
- & - & 1 & 1 & 1 \\
- & - & - & 1 & 1 \\
- & - & - & - & 1
\end{bmatrix}
\]
Then $P(b)$ is the set of solutions of the equations 
\[
TLx=T(\begin{smallmatrix}1\\b\end{smallmatrix}), \qquad  x(\lambda)\geq 0
\]
and one checks that $L'=TL$ is the matrix whose rows are the basis elements $\ell'_0,\dots,\ell'_4$ and by inspection of Definition \ref{D:bprime}, $T(\begin{smallmatrix}1\\b\end{smallmatrix})$ is the vector $b'$. 
Thus, $P(b)$ is the solution set of
\[
L'x=b'=\frac{1}{2}\begin{pmatrix} 1-b(1) \\ b(2)-b(1) \\ b(3)-b(2) \\ b(4)-b(3) \\ 1+b(4) \end{pmatrix} \qquad \text{ and } \qquad x(\lambda) \geq 0.
\]
Compare with Proposition \ref{P:P(b) in basis ell'}.
The matrix $L'$ has the form
\[
L'=
\left[
\begin{array}{cccccccccccccccc}
\scriptscriptstyle{\mu_4} &
\scriptscriptstyle{\mu_3} &
                    & 
\scriptscriptstyle{\mu_2} &
                    &
                    &
                    &
\scriptscriptstyle{\mu_1} &
                    &
                    &
                    &
                    &
                    &
                    &
                    &
\scriptscriptstyle{\mu_0} 
\\
{\scriptscriptstyle 0000} &
{\scriptscriptstyle 0001} &
{\scriptscriptstyle 0010} &
{\scriptscriptstyle 0011} &
{\scriptscriptstyle 0100} &
{\scriptscriptstyle 0101} &
{\scriptscriptstyle 0110} &
{\scriptscriptstyle 0111} &
{\scriptscriptstyle 1000} &
{\scriptscriptstyle 1001} &
{\scriptscriptstyle 1010} &
{\scriptscriptstyle 1011} &
{\scriptscriptstyle 1100} &
{\scriptscriptstyle 1101} &
{\scriptscriptstyle 1110} &
{\scriptscriptstyle 1111} \\
0 & 0 & 0 & 0 & 0 & 0 & 0 & 0 & 1 & 1 & 1 & 1 & 1 & 1 & 1 & 1 \\
0 & 0 & 0 & 0 & 1 & 1 & 1 & 1 & - & - & - & - & 0 & 0 & 0 & 0 \\
0 & 0 & 1 & 1 & - & - & 0 & 0 & 0 & 0 & 1 & 1 & - & - & 0 & 0 \\
0 & 1 & - & 0 & 0 & 1 & - & 0 & 0 & 1 & - & 0 & 0 & 1 & - & 0 \\
1 & 0 & 1 & 0 & 1 & 0 & 1 & 0 & 1 & 0 & 1 & 0 & 1 & 0 & 1 & 0 
\end{array}
\right]
\]
Observe that the entries of $L'$ are $\pm 1$ or $0$.
In each column the non-zero entries form a sequence of alternating $1$'s and $-1$'s, starting and ending with $1$.
Also notice that the columns indexed by $\mu_0,\dots,\mu_4$ form the standard basis of $R^{4+1}$.
These facts are not a coincidence and will play a major role.
We now turn to prove these crucial facts.
\end{example}

\begin{defn}\label{D:c_lambda}
For any $\lambda \in \L$ let $c_\lambda$ be the vector in $\RR^{m+1}$ defined by
\[
c_\lambda(i) = \ell_i'(\lambda), \qquad 0 \leq i \leq m.
\]
\end{defn}

The vectors $c_\lambda$ are the columns of the matrix $L'$ in Example \ref{EX:L and L'}.

\begin{lemma}\label{L:c_lambda explicit}
Let $\lambda \in \L$.
With the convention $\lambda(0)=0$ and $\lambda(m+1)=1$ in Section \ref{SS:L}, 
\[
c_\lambda(i) = \lambda(i+1)-\lambda(i), \qquad i=0,\dots,m.
\]
\end{lemma}

\begin{proof}
For any $a,b \in \{0,1\}$ we have $\frac{1}{2}((-1)^a - (-1)^b)=b-a$.
If $i=0,\dots,m-1$ then $c_\lambda(i)=\ell'_i(\lambda)=(\ell_i(\lambda)-\ell_{i+1}(\lambda))/2=((-1)^{\lambda(i)}-(-1)^{\lambda(i+1)})/2=\lambda(i+1)-\lambda(i)$.
If $i=m$ then $c_\lambda(m)=\ell'_m(\lambda) = (\ell_0(\lambda)+\ell_m(\lambda))/2=((-1)^{\lambda(m)}-(-1)^{\lambda(m+1)})/2=\lambda(m+1)-\lambda(m)$.
\end{proof}

\begin{cor}\label{C:c lambda support}
For any $\lambda \in \L$ the values of $c_\lambda$ are either $0, 1$ or $-1$.
If $0 \leq i_1<i_2<\dots<i_k \leq m$ are the indices for which $c_\lambda(i)\neq 0$ then $k$ is odd and $c_\lambda(i_1), c_\lambda(i_2),\dots,c_\lambda(i_k)$ is a sequence of the form $1,-1,1,-1,1,\dots,-1,1$ of alternating $1$'s and $-1$'s.
\end{cor}

\begin{proof}
Since $\lambda$ only attains the values $0,1$ and $\lambda(0)=0$ and $\lambda(m+1)=1$, it is clear the the sequence of differences $c_\lambda(i)=\lambda(i+1)-\lambda(i)$ must consist of only $0$ and $\pm 1$, and its support must starts at $1$ (because $\lambda(0)=0)$, end at $1$ (because $\lambda(m+1)=1$) and is alternating between $1$ and $-1$ (or else $\lambda(i)\neq 0,1$ for some $i$).
\end{proof}

Next, we single out a set of $m+1$ elements  $\mu_0,\dots,\mu_m \in \L$.
See Example \ref{EX:L and L'}.

\begin{defn}\label{D:mu_i}
For $0 \leq i \leq m$ define $\mu_i \in \L=2^{[m]}$ by
\[
\mu_i(k) = \left\{
\begin{array}{ll}
0 & \text{if } k \leq i \\
1 & \text{if } k \geq i+1
\end{array}\right.
\qquad (1 \leq k \leq m)
\]
\end{defn}
Thus, $\mu_i \in \L$ can be described as the following vectors
\[
\mu_i = (\underbrace{0,\dots,0}_{i \text{ times}},\underbrace{1,\dots,1}_{m-i \text{ times}}). 
\]
By definition of the partial order on $\L$ in Section \ref{SS:L}
\[
\mathbf{1}=\mu_0 \succ \mu_1 \succ \dots \succ \mu_m=\mathbf{0}
\]
where $\mathbf{0}$ and $\mathbf{1}$ are the minimal and maximal elements of the poset $\L$. 

\begin{lemma}\label{L:c(mu_i)=e_i}
For any $0 \leq i \leq m$ the vector $c_{\mu_i}$ is the standard basis vector $e_i \in \RR^{m+1}$.
\end{lemma}

\begin{proof}
By Definition \ref{D:mu_i} and Lemma \ref{L:c_lambda explicit}, $c_{\mu_i}(j) = \mu_i(j+1)-\mu_i(j)=\delta_{i,j}$ for all $0 \leq j \leq m$.
\end{proof}

Recall that for any $\lambda \in \L$ we denote by $e_\lambda \in \RR^\L$ the standard basis vector $e_\lambda(\lambda')=\delta_{\lambda,\lambda'}$.

\begin{defn}[The supervertex - the decreasing case]\label{D:supervertex - descending case}
Consider $b \in \RR^m$ such that $\|b\|_\infty \leq 1$.
Assume that $b$ is decreasing, namely 
\[
b(1) \geq b(2) \geq \dots \geq b(m)
\]
Recall $b'$ from Definition \ref{D:bprime}.
The {\em supervertex} of $P(b)$ is the vector $q^* \in \RR^\L$ defined by
\[
q^*= \sum_{i=0}^m b'(i) \cdot e_{\mu_i},
\]
\end{defn}

Definition \ref{D:supervertex - descending case} requires justification: A-priori it is not clear that $q^* \in P(b)$ and that it is a vertex of this polytope.
This is the content of Proposition \ref{P:supervertex is a vertex} below.

Since the facets of $\Delta(\L)$ are contained in the hyperplanes $H_{{\lambda_0}}=\{ x \in \RR^\L : x(\lambda_0)=0\}$ and since $P(b)$ is the intersection of $\Delta(\L)$ with hyperplanes in $\RR^\L$, it easily follows that $x \in P(b)$ is a vertex if and only if $\supp(x)$ is minimal with respect to inclusion, namely no $y \in P(b)$ has $\supp(y) \subsetneq \supp(x)$.

Recall that we agree that $b(0)=1$ and $b(m+1)=-1$. 
Therefore, by definition of $b'$,  
\begin{equation}\label{E:supp q_b}
\supp(q^*) = \{ \mu_i \ : \ \text{$0 \leq i \leq m$ \  and \ $b(i)>b(i+1)$}\}.
\end{equation}

\begin{prop}\label{P:supervertex is a vertex}
The vector $q^* \in \RR^\L$ from Definition \ref{D:supervertex - descending case} is a vertex of $P(b)$.
\end{prop}

\begin{proof}
First, $P(b)$ is not empty by Proposition \ref{P:non empty P(b)}.
Since  $b$ is decreasing and $\|b\|_\infty \leq 1$ we see that $b'(i) \geq 0$ for all $i$.
Therefore $q^*(\lambda) \geq 0$ for all $\lambda \in \L$.
Next, for any $0 \leq k \leq m$, Lemma \ref{L:c(mu_i)=e_i} implies that
\[
\langle \ell_k',q^* \rangle =
\sum_{i=0}^m b'(i) \langle \ell'_k , e_{\mu_i} \rangle =
\sum_{i=0}^m b'(i) \ell_k'(\mu_i) =
\sum_{i=0}^m b'(i) c_{\mu_i}(k) =
b'(k).
\]
Proposition \ref{P:P(b) in basis ell'} shows that $q^* \in P(b)$.

To show that $q^*$ is a vertex, consider $x \in P(b)$ such that $\supp(x) \subseteq \supp(q^*)$.
Set $y=x-q^*$.
Then $\supp(y) \subseteq \{ \mu_0,\dots,\mu_m\}$ and $\langle \ell'_i , y \rangle =0$ for all $0 \leq i \leq m$ by Proposition \ref{P:P(b) in basis ell'}.
By Lemma \ref{L:c(mu_i)=e_i},
\[
\langle \ell'_i,y \rangle = 
\sum_{j=0}^m y(\mu_j)\ell'_i(\mu_j) = 
\sum_{j=0}^m y(\mu_j)c_{\mu_j}(i) = y(\mu_i).
\]
This shows that $y=0$, hence $x=q^*$ as needed.
\end{proof}


Recall the action of $\Sigma_m$ on $\RR^m$ and $\RR^\L$ from Section \ref{SS:action Sigma_m on R^L}.

\begin{defn}[The supervertex - the general case]
\label{D:supervertex - general case}
Let $b \in \RR^m$ be such that $\|b\|_\infty \leq 1$.
Choose $\sigma \in \Sigma_m$ such that $\sigma_*(b)$ is decreasing.
The {\em supervertex} $q_b^*$ of $P(b)$ is the preimage of the supervertex $q_{\sigma_*(b)}^* \in P(\sigma_*(b))$ in Definition \ref{D:supervertex - descending case} under the linear isomorphism of polytopes $\sigma_* \colon P(b) \to P(\sigma_*(b))$ in Proposition \ref{P:P(b)-->P(sigma(b))}.
\end{defn}

Once again, we need to justify the definition.
A-priori it is not clear that the definition of $q^*$ is independent of the choice of $\sigma$, thus making supervertices possibly non-unique.

\begin{prop}
Let $b \in \RR^m$ satisfy $\|b\|_\infty \leq 1$ so that $P(b)$ is not empty.
Then the vertex $q_b^* \in P(b)$ from Definition \ref{D:supervertex - general case} is independent of the choice of $\sigma$.
Moreover, $\sigma_*(q_b^*)=q_b^*$ for any $\sigma \in \Sigma_m$ such that $\sigma_*(P(b))=P(b)$.
\end{prop}
\begin{proof}
Suppose that $\tau \in \Sigma_m$ is another permutation with $\tau_*(b)$ decreasing.
Then $\tau_*(b)=\sigma_*(b)$ and in particular $P(\sigma_*(b))=P(\tau_*(b))$ and $q_{\sigma_*(b)}^*=q_{\tau_*(b)}^*$ in Definition \ref{D:supervertex - descending case}.
Thus, the proof of the proposition reduces to showing that for any {\em decreasing} $b \in \RR^m$, the supervertex $q^* \in P(b)$ in Definition \ref{D:supervertex - descending case} is fixed by the linear isomorphism $\sigma_* \colon P(b) \to P(b)$  for any $\sigma \in \Sigma_m$ such that $\sigma_*(b)=b$.
For the remainder of the proof we fix such decreasing $b$ and such $\sigma$.
The claim that $\sigma_*(q^*)=q^*$ will follows once we show that $\sigma_*^{-1}(\lambda)=\lambda$ for any $\lambda \in \supp(q^*)$.

Suppose $\mu_i \in \supp(q^*)$, see \eqref{E:supp q_b}.
Since $\sigma_*^{-1}(b)=b$ it is clear that $\sigma$ acts by permuting the sets of indices $j \in [m]$ for which the values of $b$ are equal.
Since $b$ is decreasing and $b'(i)=q^*(\mu_i)>0$, if $j \geq i+1$ then $b(\sigma(j))=b(j)>b(i)$ so $\sigma(j) \geq i+1$. 
Thus, $\sigma$ permutes $\{i+1,\dots,m\}$ and $\{1,\dots,i\}$ separately.
It follows  directly from Definition \ref{D:mu_i} that $\sigma_*^{-1}(\mu_i)=\mu_i$ as needed.
\end{proof}

\begin{void}\label{V:direct description of supervertex}
{\bf Direct description of the supervertex.}
Given $b \in \RR^m$ with $\|b\|_\infty \leq 1$ we can describe the supervertex $q^*$ of $P(b)$ as follows.
As above, it is understood that $b(0)=1$ and $b(m+1)=-1$.
We can arrange $b(1),b(2),\dots,b(m)$ in decreasing order
\[
b(i_1) \geq b(i_2) \geq \dots \geq b(i_m).
\]
For $k=0,\dots,m$ let $\theta_k \in \L$ be the  characteristic function of $\{i_{k+1}, i_{k+2}, \dots, i_m\} \subseteq [m]$.
The supervertex has the form
\[
q^*=\sum_{k=0}^m \tfrac{b(i_k)-b(i_{k+1})}{2} \cdot e_{\theta_k},
\]
where it is understood that $i_0=0$ and $i_{m+1}=m+1$ and $b(0)=1$ and $b(m+1)=-1$.
\end{void}

\section{The subvertex of $P(b)$}

\begin{defn}\label{D:nu_i}
Define the following elements of $\nu_0,\dots,\nu_m \in \L$.
For any $0 \leq i \leq m$
\[
\nu_i(j)=1-\delta_{ij} \qquad  (1 \leq j \leq m). 
\]
\end{defn}

Thus, $\nu_0=\mathbf{1}$ is the maximal element of $\L$, and $\nu_i = (1,\dots,1,0,1,\dots,1)$ where the $0$ is at the $i$th position.

\begin{defn}\label{D:b''}
For $b \in \RR^m$ define $b'' \in \RR^{m+1}$ by
\begin{eqnarray*}
&& b''(i)=\frac{b(i)+1}{2} \qquad \text{ for $1 \leq i \leq m$, and } \\
&& b''(0)=1-\sum_{i=1}^m b''(i).
\end{eqnarray*}
\end{defn}

\begin{prop}
By construction $\sum_{i=0}^m b''(i)=1$.
If $\|b\|_\infty \leq 1$ then $b''(i) \geq 0$ for all $1 \leq i \leq m$ (but $b''(0)$ may be negative).
\hfill $\Box$
\end{prop}

\begin{defn}\label{D:subvertex}
Suppose that $\|b\|_\infty \leq 1$.
The {\em subvertex} of $P(b)$ is $q^* \in \RR^\L$ defined by
\[
q_* = \sum_{i=0}^m b''(i) \cdot e_{\nu_i}.
\]
\end{defn}

\noindent
{\bf Important remark:} 
As its name suggests, as well as deceives, $q_*$ {\em need not} be an element of $P(b)$.
Remarkably, by Proposition \ref{P:q_* in P(b)} below, if $q_* \in P(b)$ then it is a vertex of this polytope

\begin{defn}\label{D:ell''}
Define vectors $\ell''_0,\dots,\ell_m'' \in \RR^\L$ by
\[
\ell''_0=\ell_0 \qquad \text{and} \qquad 
\ell''_i=\tfrac{1}{2}(\ell_i+\ell_0) \text{ for all $1 \leq i \leq m$.}
\]
\end{defn}

The next lemma follows directly from Definitions \ref{D:P(b)}, \ref{D:b''} and \ref{D:ell''}.

\begin{lemma}\label{L:P(b) in terms of ell''}
Let $x \in \RR^\L$.
Then $x \in P(b)$ if and only if $\langle \ell''_0,x\rangle =1$ and $\langle \ell''_i,x\rangle =b''(i)$ for all $1 \leq i \leq m$ and $x(\lambda) \geq 0$ for all $\lambda \in \L$.
\hfill $\Box$.
\end{lemma}

\begin{lemma}\label{L:explicit ell''}
With the convention $\lambda(0)=0$ for all $\lambda \in \L$ in Section \ref{SS:L}, for any $0 \leq i \leq m$
\[
\ell''_i(\lambda)=1-\lambda(i).
\]
In particular $\ell''_0(\nu_j)=1$ and $\ell''_i(\nu_j)=\delta_{ij}$ for all $1 \leq i \leq m$ and all $0 \leq j \leq m$.
\end{lemma}
\begin{proof}
If $a=0,1$ then $\frac{1+(-1)^a}{2}=1-a$.
\end{proof}

\begin{prop}\label{P:q_* in P(b)}
Suppose that $\|b\|_\infty \leq 1$.
Then $q_*$ belongs to $P(b)$ if and only if $\tfrac{1}{m}\sum_{i=1}^m b(i) \leq \tfrac{2}{m}-1$.
In this case $q_*$ is in fact a vertex of the polytope $P(b)$.
\end{prop}

\begin{proof}
Notice that $\langle \ell''_i , q_* \rangle = \sum_{j=0}^m b''(j) \ell''_i(\nu_j)$ for all $i$.
It follows from Lemma \ref{L:explicit ell''} and the definition of $b''$ that $\langle \ell''_0 , q_* \rangle=\sum_{j=0}^m b''(j)=1$ and that $\langle \ell''_i , q_* \rangle=b''(i)$ for all $1 \leq i \leq m$.
Since $\|b\|_\infty \leq 1$ it follows that $b''(i) \geq 0$ for all $i \geq 1$, so Lemma \ref{L:P(b) in terms of ell''} implies that $q_* \in P(b)$ if and only if $b''(0) \geq 0$.
By Inspection of Definition \ref{D:b''}, this is equivalent to the requirement $\sum_{i=1}^m b(i) \leq 2-m$, as needed.

It remains to prove that $q_*$ is a vertex of $P(b)$ in this case.
Suppose that $x \in P(b)$ and $\supp(x) \subseteq \supp(q_*)$.
Set $y=x-q_*$.
Then $\supp(y) \subseteq \supp(q_*) \subseteq \{\nu_0,\dots,\nu_m\}$ and by Lemma \ref{L:P(b) in terms of ell''}, $\langle \ell''_i, y\rangle =0$ for all $i \geq 0$.
Clearly $\langle \ell''_i,y \rangle = \sum_{j=0}^m y(\nu_j)\ell''_i(\nu_j)$.
By Lemma \ref{L:explicit ell''}, if $i \geq 1$ we have $y(\nu_i)=\langle \ell''_i,y \rangle = 0$.
If $i=0$ then $0=\langle \ell''_0,y\rangle = \sum_{j=0}^m y(\nu_j)$ so $y(\nu_0)=0$ as well.
Hence $y=0$, so $x=q_*$, and therefore $q_*$ is a vertex of $P(b)$.
\end{proof}

\section{Truncation, $\ell$-positive vectors, maximum and minimum}

Throughout this section we assume that $P(b)$ is not empty, i.e $\|b\|_\infty \leq 1$.
Recall $\ell_0,\dots,\ell_m$ and $U$ from Definition \ref{D:ell_i and U} and the sets $\Ulpos$ and $(\Ulpos)^+$ of (truncated) $\ell$-positive vectors defined in Section \ref{SS:ell positive}.
The purpose of this section is to prove the following theorems.

\begin{theorem}\label{T:max at supervertex}
Let $q^*$ be the supervertex of $P(b)$.
Then for any $u \in \Ulpos$
\[
\max \{ \langle u^+,x \rangle : x \in P(b) \} = \langle u^+,q^* \rangle.
\]
\end{theorem}

\begin{theorem}\label{T:min at subvertex}
Let $q_*$ be the subvertex of $P(b)$.
Then $\langle u^+,q_* \rangle \geq 0$ for any $u \in \Ulpos$, and
\[
\min \{ \langle u^+,x \rangle : x \in P(b) \} \geq \langle u^+,q_* \rangle.
\]
If $\tfrac{1}{m} \sum_{i=1}^m b(i) \leq \frac{2}{m}-1$ then $q_*$ is a vertex of $P(b)$ and the inequality is an equality.
\end{theorem}

Thus, {\em all} the functions $f(x)=\langle u^+,x\rangle$, where $u \in \Ulpos$, attain their maximum on $P(b)$ at the supervertex.
It can be shown by means of examples that different such functions $f$ attain their minimum at different vertices of $P(b)$, so Theorem \ref{T:min at subvertex} is as strong as can be.
The fact that the minimum is attained uniformly at the vertex $q_*$ for all polytopes $P(b)$ for which $b$ belongs to a neighbourhood of the corner $(-1,\dots,-1)$ of the cube $[-1,1]^m$ is very surprising.

\begin{cor}\label{C:sum of ell-pos+ at supervertex}
Let $q^*$ and $q_*$ be the supervertex and the subvertex of $P(b)$.
Let $u_1,\dots,u_n \in U_{\ell\text{-pos}}$ and set $f=\sum_{i=1}^n u_i^+$.
Then
\begin{eqnarray*}
&& \max \{ \langle f,x \rangle \ : x \in P(b) \} \ = \ \langle f , q^* \rangle \\
&& \min \{ \langle f,x \rangle \ : x \in P(b) \} \ \geq \ \langle f, q_* \rangle.
\end{eqnarray*}
If $\frac{1}{m} \sum_{i=1}^m b(i) \leq \tfrac{2}{m}-1$ then $q_* \in P(b)$ and equality holds.
\hfill $\Box$
\end{cor}

\begin{prop}\label{P:Ulpos closure to scalars and addition}
The set $\Ulpos$ is closed under addition of vectors and multiplication by positive scalars.
Also, if $u \in \Ulpos$ then $u+c\ell_0 \in \Ulpos$ for any $c \in \RR$.
The set $(\Ulpos)^+$ is closed under addition of vectors and  multiplication by positive scalars.
\end{prop}
\begin{proof}
Immediate from the definitions.
\end{proof}

Recall the partial order $\preceq$ on $\L$, see Section \ref{SS:L}.

\begin{lemma}\label{L:supp u closed down}
Let $u \in \Ulpos$.
Then as a function $u \colon \L \to \RR$ it is order reversing, i.e $u(\lambda) \leq u(\lambda')$ if $\lambda' \preceq \lambda$.
In particular $\supp(u^+) \subseteq \L$ is closed downwards with respect to $\preceq$.
That is, if $\lambda \in \supp(u^+)$ and $\lambda' \preceq \lambda$ then $\lambda' \in \supp(u^+)$.
\end{lemma}

\begin{proof}
For any $\lambda \in \L$
\[
u(\lambda) = a_0+\sum_{i=1}^m a_i\ell_i(\lambda) = a_0+\sum_{i=1}^m (-1)^{\lambda(i)}a_i =
a_0 + \sum_{i \notin \supp(\lambda)} a_i - \sum_{i \in \supp(\lambda)} a_i.
\]
Since $a_1,\dots,a_m >0$ it is clear that if $\lambda' \preceq \lambda$, i.e $\supp(\lambda') \subseteq \supp(\lambda)$, then $u(\lambda') \geq u(\lambda)$.
\end{proof}

Recall the bases  $\ell'_0,\dots,\ell_m'$ and $\ell''_0,\dots,\ell''_m$ of $U$ from Definitions \ref{D:ell_i and U} and \ref{D:ell''}.

\begin{lemma}\label{L:ell-positive in ell'}
Let $u=\sum_{i=0}^m \alpha_i \ell'_i$ be an element of $U$.
Then $u \in \Ulpos$ if and only if $\alpha_0<\dots< \alpha_m$.
\end{lemma}

\begin{proof}
When presented $u=\sum_{i=0}^m a_i \ell_i$, one checks that $a_0=\tfrac{\alpha_0+\alpha_m}{2}$ and $a_i=\tfrac{\alpha_{i}-\alpha_{i-1}}{2}$ for $i=1,\dots,m$.
The lemma follows.
\end{proof}

\begin{lemma}\label{L:ell-positive in ell''}
Let $u=\sum_{i=0}^m \alpha_i \ell''_i$ be an element of $U$.
Then $u \in \Ulpos$ if and only if $\alpha_1,\dots,\alpha_m>0$ (and no condition on $\alpha_0$).
\end{lemma}

\begin{proof}
Written $u=\sum_{i=0}^m a_i \ell_i$, one checks that $\alpha_i=2a_i$ and $\alpha_0=a_0-\sum_{i=1}^ma_i$.
\end{proof}

\begin{example}
Let us illustrate the proof of Theorem \ref{T:max at supervertex} when $m=4$.
Let $b \in \RR^4$ be such that $|b(i)| \leq 1$.
Assume further that $b$ is decreasing, i.e $b(1) \geq b(2) \geq b(3) \geq b(4)$.
We have seen that $P(b)$ is the solution set of the equations
\[
L'x=b' \qquad \text{and} \qquad x(\lambda) \geq 0
\]
where the rows of the matrix $L'$ are the vectors $\ell'_0,\dots,\ell_4'$ in Definition \ref{D:define ell'}.
In the present example we will refer to Example \ref{EX:L and L'} where we describe $L'$ and $b'$ explicitly.
The supervertex of $P(b)$ has the form $q^*=\sum_{i=0}^4 b'(i) \cdot e_{\mu_i}$, see Definitions \ref{D:bprime} and \ref{D:mu_i} and \ref{D:supervertex - descending case}.

Let $u=\sum_{i=0}^4 \alpha_i \ell_i'$ be $\ell$-positive.
By Lemma \ref{L:ell-positive in ell'}, $\alpha_0<\dots<\alpha_4$ and there is  $k$ such that 
\[
\alpha_0 < \dots <\alpha_{k-1} \leq \ 0 \ < \alpha_k < \dots < \alpha_4.
\]
We will write $\alpha=(\alpha_0,\dots,\alpha_4)$ for the row vector in $\RR^5$.
Then $u=\alpha \cdot L'$ and $u^+=\alpha \cdot L'_{\supp(u^+)}$ where $L'_{\supp(u^+)}$ is the matrix obtained from $L'$ by setting to zero the $\lambda$-th columns for all $\lambda \notin \supp(u^+)$.
Thus, for any $x \in P(b)$ considered as a column vector,
\[
\langle u^+,x \rangle = \alpha \cdot L'_{\supp(u^+)} \cdot x.
\]
Since the $\mu_i$-th column of $L'$ is the standard basis vector $e_i$, see Example \ref{EX:L and L'} and compare with Lemma \ref{L:c(mu_i)=e_i}, $u(\mu_i)=\alpha_i$ so $\mu_i \in \supp(u^+) \iff i \geq k$.
Since $q^*$ is supported by $\mu_0,\dots,\mu_4$
\[
\langle u^+,q^* \rangle = 
\alpha \cdot (L'_{\supp(u^+)} \cdot q_*) =
\alpha \cdot (0,\dots,0,b'(k),\dots ,b'(4)).
\]
Write $\alpha_-=(\alpha_0,\dots,\alpha_{k-1},0,\dots,0)$ and $\alpha_+ = (0,\dots,0,\alpha_k,\dots,\alpha_4)$.
Then $\alpha=\alpha_-+\alpha_+$.
Observe that the non-zero entries of each  column of $L'$ form a sequence $1,-1,1,\dots,1$ of alternating $\pm 1$, compare with Corollary \ref{C:c lambda support}.
We claim that all the entries of the (row) vector
\[
(*) \qquad \qquad \alpha_- \cdot L'
\]
are non-positive.
Indeed, its $\lambda$-th entry is equal to the product of $\alpha_-$ with the $\lambda$-th column of $L'$, which has the form $\alpha_{i_1} - \alpha_{i_2}+\alpha_{i_3}-\cdots\pm\alpha_{i_t}$ where $0\leq i_1<i_2<\dots < i_t\leq k-1$.
Since $\alpha$ is increasing and $\alpha_i \leq 0$ for $i \leq k-1$, collecting the terms in pairs shows that this is a sum of negative numbers (if $t$ is even) and possibly a non-positive last term $\alpha_{i_t}$ (if $t$ is odd).
Similarly, we claim that all the entries of the vector
\[
(**) \qquad \qquad \alpha_+ \cdot L'
\]
are non-negative.
The $\lambda$-th entry is the product of $\alpha_+$ with the $\lambda$-th column of $L'$ which has the form $\alpha_{i_t}-\alpha_{i_{t-1}} + \dots \pm \alpha_{i_1}$ where $k \leq i_1< \dots < i_t \leq n$.
Since $\alpha$ is increasing and $\alpha_i >0$ for $i \geq k$, collecting terms in pairs starting from the last term shows that this is the sum of positive numbers and possibly a positive first term $\alpha_{i_1}$. 

Finally, suppose that $x \in P(b)$.
Then $x(\lambda) \geq 0$ for all $\lambda$ and $L'x=b'$.
Then 
\begin{multline*}
\langle u^+,x \rangle = 
\alpha \cdot L'_{\supp(u^+)} \cdot x = \alpha_- \cdot L'_{\supp(u^+)} \cdot x + \alpha_+ \cdot L'_{\supp(u^+)} \cdot x 
\stackrel{(*)}{\leq }
\alpha_+ \cdot L'_{\supp(u^+)} \cdot x 
\\
\stackrel{(**)}{\leq }
\alpha_+ \cdot L' \cdot x = 
(0,\dots,0,\alpha_k,\dots,\alpha_4) \cdot b' = \alpha \cdot (0,\dots,0,b'(k),\dots,b'(n)) =
\langle u^+,q^*\rangle.
\end{multline*}
\end{example}

\begin{proof}[Proof of Theorem \ref{T:max at supervertex}]
First, choose some $\sigma \in \Sigma_m$ such that $\sigma_*(b)$ is decreasing.
Since $\sigma_*$ acts by permuting the factors of $\RR^\L$ it is clear that $\sigma_*(u^+)=\sigma_*(u)^+$.
Also, $\sigma_*$ is an orthogonal transformation so for any $x \in P(b)$
\[
\langle u^+,x \rangle = 
\langle \sigma_*(u^+),\sigma_*(x) \rangle = 
\langle \sigma_*(u)^+,\sigma_*(x) \rangle.
\]
It follows from Proposition \ref{P:Sigma_m permutes ell_i} that $\sigma_*(u)$ is $\ell$-positive.
Since $\sigma_* \colon P(b) \to P(\sigma_*(b))$ is a linear homeomorphism, we may replace $b$ with $\sigma_*(b)$ and $P(b)$ with $P(\sigma_*(b))$ and $u$ with $\sigma_*(u)$.
So for the rest of the proof we assume that $b$ is decreasing.
Also, to avoid triviality we assume that $u^+ \neq 0$, namely $\supp(u^+) \neq \emptyset$.
By Lemma \ref{L:ell-positive in ell'},
\[
u=\sum_{i=0}^m \alpha_i \ell_i' \qquad \text{where} \qquad \alpha_0<\alpha_1<\dots<\alpha_m.
\]
By definition of the elements $\mu_0,\dots,\mu_m \in \L$ we have $\mu_0 \succeq \mu_1 \succeq \dots \succeq \mu_m$.
Since $\supp(u^+) \neq \emptyset$ and $\mu_m$ is the minimum of $\L$, it follows from Lemma \ref{L:supp u closed down}  that there is a smallest index $k$ such that $\mu_k \in \supp(u^+)$. 
Thus, 
\[
\mu_i \in \supp(u^+) \iff i \geq k.
\]
By Lemma \ref{L:c(mu_i)=e_i} and Definition \ref{D:c_lambda} (of $c_\lambda$) 
\[
u(\mu_j) = 
\sum_{i=0}^m \alpha_i\ell_i'(\mu_j) = 
\sum_{i=0}^m \alpha_i c_{\mu_j}(i) =
\sum_{i=0}^m \alpha_i \delta_{i,j} =
\alpha_j.
\]
By the choice of $k$ we get that $u^+(\mu_i)=0$ if and only if $0 \leq i \leq k-1$ and that 
\[
\alpha_0 < \alpha_1 < \dots < \alpha_{k-1} \ \leq 0 < \ \alpha_k < \dots < \alpha_m.
\]
Therefore, see Definition \ref{D:supervertex - descending case},
\begin{equation}\label{E:u+ dt q}
\langle u^+,q^* \rangle = 
\sum_{\lambda \in \L} u^+(\lambda)q^*(\lambda) = 
\sum_{i=0}^m u^+(\mu_i) q^*(\mu_i) = 
\sum_{i=k}^m u(\mu_i)q^*(\mu_i) = 
\sum_{i=k}^m \alpha_i b'(i).
\end{equation}
Consider some $\lambda \in \L$.
By Definition \ref{D:c_lambda} and with $k$ defined above
\begin{equation}\label{E:u(lambda) sum k of c lambda}
u(\lambda)=
\sum_{i=0}^m \alpha_i \ell_i'(\lambda) = 
\sum_{i=0}^m \alpha_i c_\lambda(i) =
\sum_{i=0}^{k-1} \alpha_i c_\lambda(i) + \sum_{i=k}^m \alpha_i c_\lambda(i).
\end{equation}
Let $I=\{i : c_\lambda(i) \neq 0\}$.
Let $I_-=I \cap \{0,\dots,k-1\}$.
By Corollary \ref{C:c lambda support}, $I_-=\{i_1<\dots<i_t\}$ and the sequence $c_\lambda(i_1),\dots,c_\lambda(i_t)$ has the form $1,-1,1,-1,\dots$.
Therefore
\[
\sum_{i=0}^{k-1} \alpha_i c_\lambda(i) = (\alpha_{i_1}-\alpha_{i_2})+(\alpha_{i_3}-\alpha_{i_4})+\cdots.
\]
Recall that $\alpha_0,\alpha_1,\dots$ is increasing.
If $t$ is even then this is a sum (possibly empty) of negative terms, and if $t$ is odd then this is a sum of negative terms and of $\alpha_{i_t}$ which is non-positive since $i_t<k$.
We deduce that 
\begin{equation}\label{E:sum c lambda up to k-1}
\sum_{i=0}^{k-1} \alpha_i c_\lambda(i) \leq 0.
\end{equation}
Set $I_+=I \cap \{k,\dots,m\}$.
By Corollary \ref{C:c lambda support}, $I_+=\{i_t<i_{t-1} < \dots < i_1\}$ where $k \leq i_t$ and $i_1 \leq m$ and the sequence $c_\lambda(i_1),\dots,c_\lambda(i_t)$ has the form $1,-1,1,-1,\dots$ and therefore
\[
\sum_{i=k}^{m} \alpha_i c_\lambda(i) = (\alpha_{i_1}-\alpha_{i_2})+(\alpha_{i_3}-\alpha_{i_4})+\cdots.
\]
If $t$ is even then this is a sum (possibly empty) of positive terms (since $\alpha_0,\alpha_1,\dots$ is increasing) and if $t$ is odd then it is a sum of positive terms and $\alpha_{i_t}$ which is also positive since $i_t \geq k$.
We deduce that 
\begin{equation}\label{E:sum c lambda after k}
\sum_{i=k}^m \alpha_i c_\lambda(i) \geq 0.
\end{equation}
Consider an arbitrary $x \in P(b)$.
By definition of $u^+$ and equation \eqref{E:u(lambda) sum k of c lambda}
\begin{multline*}
\langle u^+,x \rangle = 
\sum_{\lambda \in \L} u^+(\lambda)x(\lambda) = 
\sum_{\lambda \in \supp(u^+)} u(\lambda)x(\lambda) = 
\\
\sum_{\lambda \in \supp(u^+)} x(\lambda) \sum_{i=0}^{k-1} \alpha_i c_\lambda(i) + \sum_{\lambda \in \supp(u^+)} x(\lambda) \sum_{i=k}^{m} \alpha_i c_\lambda(i).
\end{multline*}
Since $x(\lambda) \geq 0$, equations \eqref{E:sum c lambda up to k-1} and \eqref{E:sum c lambda after k} allow us to continue the estimate of $\langle u^+,x \rangle $
\begin{multline*}
\leq 
\sum_{\lambda \in \supp(u^+)} x(\lambda) \sum_{i=k}^{m} \alpha_i c_\lambda(i)
\leq 
\sum_{\lambda \in \L} x(\lambda) \sum_{i=k}^{m} \alpha_i c_\lambda(i)
=
\sum_{i=k}^m \alpha_i \sum_{\lambda \in \L} \ell_i'(\lambda) x(\lambda) =
\\
\sum_{i=k}^m \alpha_i \langle \ell_i',x \rangle = 
\sum_{i=k}^m \alpha_i b'(i) = 
\langle u^+,q^* \rangle.
\end{multline*}
This completes the proof.
\end{proof}

\begin{proof}[Proof of Theorem \ref{T:min at subvertex}]
The last statement in the theorem follows immediately from Proposition \ref{P:q_* in P(b)} and the first statement.
So it remain to prove the inequality and that $\langle u^+,q_*\rangle \geq 0$.

By Lemma \ref{L:ell-positive in ell''}, $u=\sum_{i=0}^m \alpha_i \ell''_i$ where $\alpha_1,\dots,\alpha_m >0$.
Since $\|b\|_\infty \leq 1$ it follows from Definition \ref{D:b''} that $b''(i) \geq 0$ for all $i \geq 1$.
Suppose that $x \in \RR^\L$ is a vector such that $\langle \ell''_0,x\rangle=1$ and $\langle \ell''_i,x\rangle=b''(i)$ for all $i \geq 1$.
Then 
\[
(*) \qquad
\langle u,x\rangle = \alpha_0+\sum_{i=1}^m \alpha_i b''(i).
\]
By Definition \ref{D:subvertex} and Lemma \ref{L:explicit ell''}, $q_*$ satisfies these conditions.
By Lemma \ref{L:P(b) in terms of ell''}, this is also the case for any $x \in P(b)$.
Suppose that $\alpha_0 \geq 0$.
Then $\langle u,q_* \rangle \geq 0$ by $(*)$.
Also, $u(\lambda)=\sum_{i=0}^m \alpha_i \ell_i''(\lambda) \geq 0$ for all $\lambda$ (since $\ell''_i(\lambda) \geq 0$), so $u^+=u$.
Thus, if $x \in P(b)$ then $\langle u^+,x \rangle = \langle u,x \rangle  \stackrel{(*)}{=} \langle u,q_* \rangle = \langle u^+,q_* \rangle$.
This prove the theorem in the case $\alpha_0 \geq 0$.
So for the rest of the proof we assume that $\alpha_0<0$.

Lemma \ref{L:explicit ell''} shows that
\begin{equation}\label{E:u at nu}
u(\nu_0)=\alpha_0 \qquad \text{and} \qquad u(\nu_i)=\alpha_0+\alpha_i \text{ for $i \geq 1$}.
\end{equation}
Notice that $\nu_0 \notin \supp(u^+)$ since  we assume that $\alpha_0<0$.
Set
\[
I(u) = \{ i : \nu_i \in \supp(u^+), 0 \leq i \leq m\} 
\]
Then $I(u) \subseteq \{1,\dots,m\}$ and since $b''(i) \geq 0$ for $i \geq 1$,
\begin{equation}\label{E:u+ at subvertex}
\langle u^+,q_* \rangle =
\sum_{i=0}^m b''(i) \cdot u^+(\nu_i) = \sum_{i \in I(u)} b''(i) \cdot u(\nu_i) \geq 0.
\end{equation}
This proves the first statement of the theorem.
Recall from Section \ref{SS:L} the convention that $\lambda(0)=0$ for all $\lambda \in \L$.
For any $1 \leq i \leq m$ set 
\begin{equation}\label{E:def L(i)}
\L(i) = \{ \lambda \in \L : \lambda(i)=0\} \stackrel{\text{(Lemma \ref{L:explicit ell''})}}{=} 
\{ \lambda \in \L : \ell''_i(\lambda)=1\}.
\end{equation}
If $i \in I(u)$ and  $\lambda \in \L(i)$ then $u(\lambda)=\sum_{j=0}^m \alpha_j \ell_j''(\lambda) = \alpha_0+\alpha_i+\sum_{j \neq 0,i} \alpha_j \ell_j''(\lambda) \geq \alpha_0+\alpha_i=u(\nu_i) > 0$. 
We deduce that 
\begin{equation}\label{E:Lambda subset supp(u)}
\Lambda(u) \stackrel{\text{def}}{=} \bigcup_{i \in I(u)} \L(i) \ \subseteq \ \supp(u^+).
\end{equation}
Consider some $x \in P(b)$.
Since $x(\lambda) \geq 0$ for all $\lambda$ 
\begin{equation}\label{E:sum x(lambda) in Lambda(u)}
\sum_{\lambda \in \Lambda(u)} x(\lambda) \leq 
\sum_{i \in I(u)} \sum_{\lambda \in \L(i)} x(\lambda) =
\sum_{i \in I(u)} \langle \ell''_i,x \rangle 
\stackrel{\text{(Lemma \ref{L:P(b) in terms of ell''})}}{=} 
\sum_{i \in I(u)} b''(i).
\end{equation}
Since $\ell_0''(\lambda)=1$ by Lemma \ref{L:explicit ell''}, and since $\ell''_i(\lambda)=1 \iff \lambda \in \L(i)$ 
\begin{multline*}
\langle u^+ ,x \rangle =
\sum_{\lambda \in \supp(u^+)} u(\lambda) x(\lambda) \stackrel{\eqref{E:Lambda subset supp(u)}}{\geq}
\sum_{\lambda \in \Lambda(u)} u(\lambda) x(\lambda) =
\sum_{\lambda \in \Lambda(u)}  \sum_{i=0}^m  \alpha_i \ell_i''(\lambda)x(\lambda) =
\\
\alpha_0 \sum_{\lambda \in \Lambda(u)} x(\lambda) + \sum_{i=1}^m \alpha_i \sum_{\lambda \in \Lambda(u)} \ell_i''(\lambda)x(\lambda) \geq
\alpha_0 \sum_{\lambda \in \Lambda(u)} x(\lambda) + \sum_{i \in I(u) } \alpha_i \sum_{\lambda \in \L(i)} \ell_i''(\lambda)x(\lambda) =
\\
\alpha_0 \sum_{\lambda \in \Lambda(u)} x(\lambda) + \sum_{i \in I(u)} \alpha_i \langle \ell_i'',x \rangle.
\end{multline*}
Thanks to \eqref{E:sum x(lambda) in Lambda(u)}, \eqref{E:u at nu} and to Lemma \ref{L:P(b) in terms of ell''}, and since $\alpha_0<0$ and $b''(i) \geq 0$ for all $i \geq 1$, we can continue the estimate
\[
\geq 
\alpha_0 \sum_{i \in I(u)} b''(i) + \sum_{i \in I(u)} \alpha_i b''(i) =
\sum_{i \in I(u)} u(\nu_i) b''(i) = \langle u^+,q_* \rangle.
\]
This completes the proof of the theorem. 
\end{proof}

\section{Functions from $\L^n$}
\label{S:truncated ell positive functions}

We fix $n >0$ and an $\L$-labelled tree $\T$, see Section \ref{SS:extension F tree}.
We start this section with a simple observation about symmetric truncated $\ell$-positive functions $F \colon \L^n \to \RR$, see Definition \ref{D:symmetric L positive}.
%

\begin{prop}\label{P:truncated ell positive closed under sums and products}
The collection of (symmetric) truncated $\ell$-positive functions is closed under addition and multiplication by positive scalars.
\hfill $\Box$
\end{prop}

\begin{void}\label{V:Notation evaluation on words}
{\bf Notation.}
Let $W(\Omega)$ denote the set of words in the alphabet $\Omega$.
For any $\omega \in \Omega$ we write $\omega^k$ for the word $\omega \cdots \omega$ of length $k$.
For any $f \colon \Omega \to \RR$ and any 
$\omega_1\cdots\omega_k \in W(\Omega)$ write
\[f(w)=f(\omega_1)\cdots f(\omega_k).
\]
\end{void}

Proposition \ref{P:cooking up F from ui} below shows that European functions, see Definition \ref{D:European}, are examples of symmetric truncated $\ell$-positive functions.
In order to construct them one needs to find $u \in \Ulpos$ such that $u=u^+$.
These are easy to construct as follows.
Choose $a_1,\dots,a_m>0$ arbitrarily.
Then for any choice of sufficiently large $a_0$ the vector $u=\sum_{i=0}^m a_i \ell_i$ is $\ell$-positive and has non-negative values, i.e $u=u^+$.
Explicit examples of such vectors are given in Proposition \ref{P:ui from Di Ui} below.

\begin{prop}\label{P:cooking up F from ui}
Any European function $F \colon \L^n \to \RR$ is symmetric truncated $\ell$-positive. 
\end{prop}

\begin{proof}
By definition $F(\lambda_1\cdots\lambda_n)=(\sum_{j=1}^r s_j \cdot u_j(\lambda_1) \cdots u_j(\lambda_n) -C)^+$ where $u_1,\dots,u_r \in \Ulpos$ and $u_i^+=u_i$ and $s_j, C \geq 0$.
The symmetry of $F$ is clear.
Given $\omega, \tau \in \T$ of total length $n-1$, the function $f(\lambda) = F(\omega \lambda \tau)$ is the vector in $\RR^\L$
\[
(\sum_{j=1}^r s_j \cdot u_j(\omega)\cdot u_j(\tau) \cdot u_j -C \cdot \ell_0)^+.
\]
Since $s_j$ and $u_j(\omega), u_j(\tau)$ are non-negative, this is a vector in $(\Ulpos)^+$.
\end{proof}

\begin{prop}\label{P:ui from Di Ui}
Choose some $1 \leq i \leq m$ and $0<D_i<U_i$.
Then $u_i \in \RR^\L$ defined by
\[
u_i(\lambda) = D_i^{\lambda(i)} U^{1-\lambda(i)} =
\left\{
\begin{array}{ll}
U_i & \text{if } \lambda(i)=0 \\
D_i & \text{if } \lambda(i)=1
\end{array}\right.
\]
is an element of $\Ulpos$.
In addition $u_i(\lambda)>0$ for all $\lambda$, i.e $u_i^+=u_i$.
\end{prop}

\begin{proof}
By definition $\ell_i(\lambda)=(-1)^{\lambda(i)}$ and $\ell_0(\lambda)=1$.
One then checks that 
\[
u_i = \tfrac{U_i-D_i}{2} \ell_i + \tfrac{U_i+D_i}{2} \ell_0.
\]
The second assertion is clear since $D_i, U_i>0$.
\end{proof}

\section{Functions on trees}

Let $\T$ be an $\L$-labelled tree of height $n$, see Section \ref{SS:extension F tree}.
A function  $\Phi \colon \T^* \to \Delta(\L)$ is merely an assignment of a probability density function on the set of successors of each vertex $\tau \in \T^*$.
For any $\omega \in \T_{n-k}$ there is a canonical bijection $\L^k \cong A_\omega$ given by $\tau \mapsto \omega\tau$, see \eqref{E:def A_omega}.
We define a function $P(\Phi,A_\omega) \colon \L^k \to \RR$ by
\begin{equation}\label{D:P(Phi,A_omega)}
P(\Phi,A_\omega) (\tau) = \prod_{j=1}^k \Phi(\omega \tau_1 \cdots \tau_{j-1})(\tau_j).
\end{equation}
If $\omega$ is the empty word then $A_\omega=\L^n$ and we write $P(\Phi)$ instead of $P(\Phi,A_\emptyset)$.
See \eqref{D:P-Phi}.

\begin{prop}\label{P:P(Phi) and its properties}
Consider $\Phi \colon \T^* \to \Delta(\L)$ and $\omega \in \T_{n-k}$ where $1 \leq k \leq n$. 

\begin{enumerate}
\item 
$P(\Phi,A_\omega)$ is a probability density function on $\L^k$.

\item
\label{Item:P:P(Phi) constant}
If $q \in \Delta(\L)$ and $\Phi=q$ is the constant function then $P(\Phi,A_\omega)$ is the product density function $(\L,q)^k$ on $\L^k$. 

\item
$P(\Phi)(A_\omega)=\prod_{j=1}^{n-k} \Phi(\omega_1\cdots \omega_{j-1})(\omega_j)$ for any $\omega \in \T_{n-k}$.

\item
If $P(\Phi)(A_\omega)>0$ then $P(\Phi)(A_{\omega \lambda} | A_\omega) = \Phi(\omega)(\lambda)$ for all $\lambda \in \L$.
More generally, $P(\Phi)(B|A_\omega)=P(\Phi,A_\omega)(B)$ for any $B \subseteq A_\omega \cong \L^k$.
\label{Item:P(Phi) conditional probability}

\item
Any probability measure $P'$ on $\L^n$ has the form $P(\Phi)$ for some $\Phi \colon \T^* \to \Delta(\L)$.
\label{Item:P(Phi) gets them all}
\end{enumerate}
\end{prop}

\begin{proof}
Elementary and left to the reader.
For the last statement define $\Phi(\omega)$ by means of item (\ref{Item:P(Phi) conditional probability}) whenever $P'(A_\omega)>0$ and arbitrarily otherwise.
\end{proof}

The following construction will be fundamental.
It gives a procedure to extend a function $F \colon \L^n \to \RR$ defined on the leaves of $\T$ to the entire tree.

\begin{defn}\label{D:extension of F}
Let $F \colon \L^n \to \RR$  and $\Phi \colon \T^* \to \RR^\L$ be functions.
Define by induction on $0 \leq k \leq n$ functions $F^{(k)}_{\Phi} \colon \T_{n-k} \to \RR$ by
\[
F^{(0)}_\Phi = F
\]
and once $F^{(k)}_\Phi$ has been defined, for any $\omega \in \T_{n-k-1}=\L^{n-k-1}$ set
\[
F^{(k+1)}_\Phi(\omega) = \left\langle F^{(k)}_\Phi|_{\OP{succ}(\omega)} , \Phi(\omega) \right\rangle =
\sum_{\lambda \in \L} F^{(k)}_\Phi(\omega\lambda) \cdot \Phi(\omega)(\lambda).
\]
\end{defn}

\begin{lemma}\label{L:F^k explicit formula}
For any $0 \leq k \leq n$ and any $\omega \in \T_{n-k}$ 
\[
F^{(k)}_\Phi (\omega) = \sum_{\theta \in \L^k} F(\omega\theta) \cdot \prod_{j=1}^k \Phi(\omega\theta_1\cdots\theta_{j-1})(\theta_j).
\]
\end{lemma}

\begin{proof}
Straightforward induction on $k$.
The details are left to the reader.
\end{proof}

\begin{prop}\label{P:F^k as conditional probability}
Let $F \colon \L^n \to \RR$ and $\Phi \colon \T^* \to \Delta(\L)$ be functions.
For any $0 \leq k \leq n$ and any $\omega \in \T_{n-k}$ recall the function $F_{\omega-}$ from \eqref{E:F_omega-}.
Then
\[
F^{(k)}_\Phi(\omega) = E_{P(\Phi,A_\omega)}(F_{\omega-}).
\]
If $P(\Phi)(A_\omega)>0$ then $F^{(k)}_\Phi(\omega) = E_{P(\Phi)}(F|A_\omega)$.
\end{prop}

\begin{proof}
Apply Lemma \ref{L:F^k explicit formula} and \eqref{D:P(Phi,A_omega)}.
If $P(\Phi)(A_\omega)>0$ use Proposition \ref{P:P(Phi) and its properties}(\ref{Item:P(Phi) conditional probability}).
\end{proof}

\begin{prop}\label{P:F^k_Phi bounded by F^k_q} 
Let $q^*$ and $q_*$ be the supervertex and the subvertex of a non-empty $P(b)$.
Let $\Gamma=\Gamma(\L^n,b)$.
By abuse of notation let $q^*$ and $q_*$ denote the constant functions $\T^* \to \RR^\L$.
Let $0 \leq k \leq n$ and let $\omega \in \T_{n-k}$.
Then for any $\Phi \colon \T^* \to P(b)$ 
\[
F^{(k)}_\Phi(\omega) \leq F^{(k)}_{q^*}(\omega)
\]
If $q_* \in P(b)$, see Proposition \ref{P:q_* in P(b)}, then 
\[
F^{(k)}_\Phi(\omega) \geq F^{(k)}_{q_*}(\omega)
\]
\end{prop}

\begin{proof}
Use induction on $k$.
The case $k=0$ is a triviality because $F^{(0}_\Phi, F^{(0)}_{q^*}, F^{(0)}_{q_*} = F$ by construction.
Assume the inequalities hold for $k$.

Consider some $\omega \in \T_{n-k-1}$.
Since $q^*(\lambda) \geq 0$ for all $\lambda$, and by assumption also $q_*(\lambda) \geq 0$, it follows from Lemma \ref{L:F^k explicit formula} and Propositions \ref{P:truncated ell positive closed under sums and products} and \ref{P:Ulpos closure to scalars and addition} that the functions $f^*(\lambda)=F^{(k)}_{q^*}(\omega\lambda)$ and $f_*(\lambda)=F^{(k)}_{q_*}(\omega\lambda)$ are truncated $\ell$-positive (we remark that here it is {\em crucial} that $q^*$ and $q_*$ are constant functions $\T^* \to P(b)$).
By definition, the induction hypothesis, and the monotonicity of the expectation and Corollary \ref{C:sum of ell-pos+ at supervertex} 
\[
F^{(k+1)}_\Phi(\omega) = 
\langle F^{(k)}_\Phi|_{\OP{succ}(\omega)}, \Phi(\omega) \rangle \leq
\langle F^{(k)}_{q^*}|_{\OP{succ}(\omega)}, \Phi(\omega) \rangle \leq
\langle F^{(k)}_{q^*}|_{\OP{succ}(\omega)}, q^* \rangle = 
F^{(k+1)}_{q^*}(\omega).
\]
An identical argument (with the inequalities revered)
shows that $F^{(k+1)}_\Phi(\omega) \geq F^{(k+1)}_{q_*}(\omega)$ provided $q_* \in P(b)$.
\end{proof}

\begin{defn}\label{D:Minimizer of European}
Let $F \colon \L^n \to \RR$ be a European function (Definition \ref{D:European}) given by $u_1,\dots,u_r \in \Ulpos$ and $s_1,\dots,s_r,C \geq 0$.
Let $P(b)$ be non-empty.
For any $0 \leq i \leq m$ set
\[
\beta(0)=1 \qquad \text{and} \qquad \beta(i)=b''(i)=\tfrac{1+b(i)}{2}.
\]
For any $1 \leq j \leq r$ and any $0 \leq i \leq m$ set
\[
\alpha_j(0)=u_j(\nu_0) \qquad \text{and} \qquad \alpha_j(i)=u_j(\nu_i)-u_j(\nu_0).
\]
The {\em minimizer} of $F$ on $P(b)$ is the function $G \colon \T^* \to \RR$ defined as follows.
Using the notation in \ref{V:Notation evaluation on words} for $u_j$ and $\beta$ and $\alpha_j$, for any $1 \leq k \leq n$ and any $\omega \in \T_{n-k}$ set
\[
G^{(k)}(\omega)=\sum_{\mathbf{i} \in \{0,\dots,m\}^k} \, \beta(\mathbf{i})\big(\sum_{j=1}^r s_j \cdot \alpha_j(\mathbf{i}) \cdot u_j(\omega)-C\big)^+.
\]
\end{defn}

The final result of this section gives a lower bound, albeit generally quite poor, for the values of $F^{(k)}_\Phi$ for European functions where $\Phi \colon \T^* \to P(b)$.
Compare with Proposition \ref{P:F^k_Phi bounded by F^k_q}.

\begin{prop}\label{P:minimizer minimizes Eurpeans}
Let $F \colon \L^n \to \RR$ be European and $\Phi \colon \T^* \to P(b)$.
Let $G$ be the minimizer of $F$ on $P(b)$.
Then for any $0 \leq k \leq n$ and any $\omega \in \T_{n-k}$
\[
F^{(k)}_\Phi(\omega) \geq G^{(k)}(\omega).
\]
In particular
\[
E_{P(\Phi)}(F) = 
F^{(n)}_\Phi(\emptyset) \geq \sum_{\mathbf{i} \in \{0,\dots,m\}^n} \beta(\mathbf{i}) \big(\sum_{j=1}^r s_j \cdot \alpha_j(\mathbf{i})-C\big)^+.
\]
\end{prop}

In preparation for the  proof we make some observations.
Suppose that $F \colon \L^n \to \RR$ is a European function defined by $u_1,\dots,u_r \geq 0$.
By Proposition \ref{P:non empty P(b)} all the numbers $\beta(i)$ in Definition \ref{D:Minimizer of European} are non-negative.
Also, $\nu_0$ in Definition \ref{D:nu_i}  is the maximum element $\mathbf{1}$ of $\L$.
Therefore $u_j(\nu_i) \geq u_j(\nu_0)$ for all $1 \leq i \leq m$ by Lemma \ref{L:supp u closed down}.
In particular $\alpha_j(i) \geq 0$, and $\alpha_j(0)=u_j(\nu_0) \geq 0$ by the assumption that $u_j=u_j^+$.

\begin{lemma}\label{L:difference of truncated ell positive}
Let $a,b,c \in \RR^r$ be vectors such that $a_i,b_i,c_i \geq 0$ and $a_i \geq b_i$ for all $i=1,\dots,r$.
Let $C \geq 0$.
Then for any $u \in \RR^r$ such that $u_i \geq 0$ for all $i$
\[
\big(\sum_{i=1}^r c_ia_iu_i-C\big)^+ - \big(\sum_{i=1}^r c_ib_iu_i-C\big)^+ \ \  \geq \ \   \big(\sum_{i=1}^r c_i(a_i-b_i)u_i-C\big)^+.
\]
\end{lemma}

\begin{proof}
Denote the left and right hand sides of the inequality by $\text{LHS}$ and $\text{RHS}$.
Since all numbers in sight are non-negative, if $\sum_{i=1}^r c_ib_iu_i \geq C$ then $\sum_{i=1}^r c_ia_iu_i \geq C$, and since $C \geq 0$
\[
\text{LHS} = \sum_{i=1}^r c_i(a_i-b_i)u_i \geq \big(\sum_{i=1}^r c_i(a_i-b_i)u_i - C\big)^+ = \text{RHS}.
\]
If $\sum_{i=1}^r c_ib_iu_i < C$ then the second term in the left hand side vanishes and the inequality holds since $0 \leq  a_i-b_i \leq a_i$ for all $i$.
\end{proof}

\begin{lemma}\label{L:<G^(k),q_*> >= g^(k+1)}
Let $q_*$ be the subvertex of $P(b)$.
Let $F \colon \L^n \to \RR$ be a European function and $G$ its minimizer on $P(b)$.
Then for any $0 \leq k \leq n$ and and any $\omega \in \T_{n-k-1}$
\[
\langle G^{(k)}|_{\OP{succ}(\omega)} , q_* \rangle \geq G^{(k+1)}(\omega).
\]
\end{lemma}

\begin{proof}
We leave it to the reader to check that Lemma \ref{L:difference of truncated ell positive} together with the facts that $u_j(\omega \nu_i)=u_j(\omega)\cdot u_j(\nu_i)$ and $\alpha_j(i)=u_j(\nu_i)-u_j(\nu_0)$ and $\beta(0)=1$ imply that
\[
\sum_{i=1}^m \beta(i) \cdot \big( G^{(k)}(\omega \nu_i) - G^{(k)}(\omega \nu_0) \big) \geq G^{(k+1)}(\omega) - G^{(k)}(\omega \nu_0).
\]
Observe that $(ax)^+=a\cdot x^+$ if $a \geq 0$ and that $(\sum_i x_i)^+ \leq \sum_i x_i^+$.
Since $\beta(i)=b''(i)$ for all $i \geq 1$ and $b''(0)=1-\sum_{i=1}^m b''(i)$ (Definitions \ref{P:minimizer minimizes Eurpeans}, \ref{D:b''}) and since $q_*=\sum_{i=0}^m b''(i) \cdot e_{\nu_i}$ (Definition \ref{D:subvertex})
\begin{multline*}
G^{(k+1)}(\omega) \leq \big(1-\sum_{i=1}^m \beta(i)\big )\cdot G^{(k)}(\omega \nu_0) + \sum_{i=1}^m \beta(i) \cdot G^{(k)}(\omega \nu_i) \\ 
= \sum_{i=0}^m b''(i) \cdot G^{(k)}(\omega \nu_i) = \langle G^{(k)}|_{\OP{succ}(\omega)} , q_* \rangle
\end{multline*}
\end{proof}

\begin{proof}[Proof of Proposition \ref{P:minimizer minimizes Eurpeans}]
Use induction on $k$.
The base of induction is $F^{(0)}_\Phi = F = G^{(0)}$.
For the induction step, observe that $G^{(k)} \colon \L^{n-k} \to \RR$ is truncated $\ell$-positive by Proposition \ref{P:truncated ell positive closed under sums and products} because $\alpha_j(i), \beta_j(i), s_j$ and $u_j$ are non-negative.
The monotonicity of the expectation, Corollary \ref{C:sum of ell-pos+ at supervertex} and Lemma \ref{L:<G^(k),q_*> >= g^(k+1)} imply
\[
F^{(k+1)}_\Phi(\omega) = 
\langle F^{(k)}_\Phi|_{\OP{succ}(\omega)} , \Phi(\omega) \rangle \geq
\langle G^{(k)}|_{\OP{succ}(\omega)} , \Phi(\omega) \rangle \geq
\langle G^{(k)}|_{\OP{succ}(\omega)} , q_* \rangle \geq 
G^{(k+1)}(\omega).
\]
This completes the induction step.
The last part follows from Proposition \ref{P:F^k as conditional probability}.
\end{proof}

\section{Proofs of the results in Section \ref{S:main results}}

The following lemma is an elementary counting argument and left to the reader.

\begin{lemma}\label{L:exponential to polynomial}
Let $\Omega=\{\omega_1,\dots,\omega_r\}$ be a finite set.
We think of $\Omega^n$ as the set of words of length $n$.
Let $f \colon \Omega^n \to \RR$ be symmetric i.e $f(x_1\cdots x_n)$ does not depend on the order of the $x_i$'s.
Then
\[
\sum_{x_1,\dots,x_n \in \Omega} f(x_1\dots x_n)=
\sum_{k_1+\cdots+k_r=n} \frac{n!}{k_1!\cdots k_r!} \cdot f(\omega_1^{k_1}\cdots\omega_r^{k_r})
\]
where $\omega^k$ denotes the $k$-tuple $\omega \dots \omega$ for any $k \geq 0$.
\hfill $\Box$
\end{lemma}

\begin{proof}[Proof of Theorem \ref{T:Fmax Fmin is expectation of product measure}]
Given $\Phi \colon \T^* \to P(b)$ apply Propositions \ref{P:F^k_Phi bounded by F^k_q}, \ref{P:F^k as conditional probability} and \ref{P:P(Phi) and its properties}(\ref{Item:P:P(Phi) constant}) to $F^{(k)}_\Phi$ and $F^{(k)}_{q^*}$ and $F^{(k)}_{q_*}$.
Use the fact that $q^*,q_* \in \Gamma$ to deduce $F_{\max}(\Gamma)=E_{q^*}(F)$ and $F_{\min}(\Gamma)=E_{q_*}(F)$.
\end{proof}

\begin{proof}[Proof of Proposition \ref{P:symmetric polynomial complexity}]
Recall that $q^*=\sum_{i=0}^m b'(i) \cdot e_{\mu_i}$ and $q_*=\sum_{i=0}^m b''(i) \cdot e_{\nu_i}$.
Since 
\begin{eqnarray*}
&& E_{q^*}(F_{\omega-}) = \sum_{\tau \in \supp(q^*)^k} F(\omega \tau_1 \cdots \tau_k) \cdot q^*(\tau_1)\cdots q^*(\tau_n) \qquad \text{and} \\
&& E_{q_*}(F_{\omega-}) = \sum_{\tau \in \supp(q_*)^k} F(\omega \tau_1 \cdots \tau_k) \cdot q_*(\tau_1)\cdots q_*(\tau_n)
\end{eqnarray*}
and since $F$ is symmetric, the result follows from Lemma \ref{L:exponential to polynomial}.
\end{proof}

\begin{proof}[Proof of Theorem \ref{T:bounding Fmin below}]
Follows from Proposition \ref{P:minimizer minimizes Eurpeans}, Lemma \ref{L:exponential to polynomial} and  the definition of $\Gamma(\L^n,b)$.
\end{proof}

\begin{proof}[Proof of Proposition \ref{P:pay off is ell positive}]
Apply Definition \ref{D:European} to $u_1,\dots,u_m$ in Proposition \ref{P:ui from Di Ui} and $s_j=S_j(0)$.
\end{proof}

\begin{proof}[Proof of Theorem \ref{T:pay off rational values}]
Recall from Section \ref{Int:FM motivation} that $\Gamma_*$ is the set of risk neutral measures on $\L^n$, that $\overline{\Gamma_*}=\Gamma(\L^n,b)$ and that we assume that $\Gamma_* \neq \emptyset$.
Fix some $\omega \in \T_{n-k}$.
Choose some $P_* \in \Gamma_*$.
By Proposition \ref{P:P(Phi) and its properties}\ref{Item:P(Phi) gets them all} and since $P_*$ has no null-sets, $P_*=P(\Phi)$ for some $\Phi \colon \T^* \to P(b)$ with values in the interior of $P(b)$.

It follows from the binomial behaviour of the processes $S_i$ that $\theta$ is completely determined by the values of $S_i(t),\dots,S_m(t)$ for all $0 \leq t \leq n-k$.
Therefore the event $E_{n-k,\omega}$ described in Section \ref{Int:FM motivation} conicides with the event $A_{\theta}$ in Section \ref{SS:extension F tree}.
Equation \eqref{E:F from f by E*} together with Propositions \ref{P:F^k as conditional probability}, \ref{P:F^k_Phi bounded by F^k_q} and \ref{P:P(Phi) and its properties}\ref{Item:P:P(Phi) constant} show that
\[
H_{P_*}(n-k)(\omega) = 
E_{P_*}(F|E_{k,\omega}) =
E_{P(\Phi)}(F|A_{\theta}) =
F^{(n-k)}_{\Phi(\theta)} \leq F_{q^*}^{(n-k)}(\theta) = E_{q^*}(F_{\theta-}).
\]
Similarly, if $q_* \in P(b)$ then $H_{P_*}(n-k)(\omega) \geq E_{q_*}(F_{\theta-})$.
The result follows from Proposition \ref{P:symmetric polynomial complexity} because $q^*$ and $q_*$ are limit points of the interior of $P(b)$
\end{proof}

\bibliography{bibliography}
\bibliographystyle{plain}

\end{document}